\documentclass[a4paper,11pt,reqno]{amsart}
\usepackage[left=26mm,right=26mm,top=30mm,bottom=30mm]{geometry}

\usepackage{amssymb}
\usepackage[shortlabels]{enumitem}
\usepackage{bm}
\usepackage{mathtools}
\usepackage{array}  
\usepackage{xcolor}  
\usepackage{xparse}
\usepackage[sort&compress,numbers]{natbib}
\usepackage{caption}
\newtheorem{theorem}{Theorem}[section]
\theoremstyle{plain}

\newtheorem{lemma}{Lemma}[section]

\newtheorem{proposition}{Proposition}[section]
\newtheorem{remark}{Remark}[section]
\numberwithin{equation}{section}

\makeatletter
\newcommand\makebig[2]{%
  \@xp\newcommand\@xp*\csname#1\endcsname{\bBigg@{#2}}%
  \@xp\newcommand\@xp*\csname#1l\endcsname{\@xp\mathopen\csname#1\endcsname}%
  \@xp\newcommand\@xp*\csname#1r\endcsname{\@xp\mathclose\csname#1\endcsname}%
}
\makeatother

\makebig{biggg} {3.0}
\makebig{Biggg} {3.5}
\makebig{bigggg}{4.0}
\makebig{Bigggg}{4.5}
\makebig{biggggg}{5}
\makebig{Biggggg}{5.5}

\makeatletter
\newsavebox{\@brx}
\newcommand{\llangle}[1][]{\savebox{\@brx}{\(\m@th{#1\langle}\)}%
  \mathopen{\copy\@brx\kern-0.5\wd\@brx\usebox{\@brx}}}
\newcommand{\rrangle}[1][]{\savebox{\@brx}{\(\m@th{#1\rangle}\)}%
  \mathclose{\copy\@brx\kern-0.5\wd\@brx\usebox{\@brx}}}
\makeatother

\begin{document}
\title{Spectral analysis of a viscoelastic tube conveying fluid with generalised boundary conditions}
\author{Xiao Xuan Feng$^{\tt1,\ast}$}
\author{Mahyar Mahinzaeim$^{\tt2}$}
\author{Gen Qi Xu$^{\tt3}$}

 \thanks{
\vspace{-1em}\newline\noindent
\newline\noindent
 $^{\tt1}$ School of Mathematics and Information Science, Hebei Normal University of Science and Technology, China.
  \newline\noindent
$^{\tt2}$ Research Center for Complex Systems, Aalen University, Germany.
    \newline\noindent
 $^{\tt3}$ Department of Mathematics, Tianjin University, China.
  \newline\noindent
  {\sc Emails}:
 {\tt 2015210135@tju.edu.cn},~{\tt m.mahinzaeim@web.de},~{\tt gqxu@tju.edu.cn}.
    \newline\noindent
$^{\ast}$ Corresponding author.
}

\begin{abstract}
We study the spectral problem associated with the equation governing the small transverse motions of a viscoelastic tube of finite length conveying an ideal fluid. The boundary conditions considered are of general form, accounting for a combination of elasticity and viscous damping acting on both the slopes and the displacements of the ends of the tube. These include many standard boundary conditions as special cases such as the clamped, free, hinged, and guided conditions. We derive explicit asymptotic formulae for the eigenvalues for the case of generalised boundary conditions and specialise these results to the clamped case and the case in which damping acts on the slopes but not on the displacements. In particular, the dependence of the eigenvalues on the parameters of the problem is investigated and it is found that all eigenvalues are located in certain sectorial sets in the complex plane.
\end{abstract}
\maketitle

\pagestyle{myheadings} \thispagestyle{plain} \markboth{\sc X.\ X.\ Feng, M.\ Mahinzaeim, G.\ Q.\ Xu}{\sc Viscoelastic tube conveying fluid with generalised boundary conditions}

\section{Introduction}\label{sec_intro}

Consider a thin homogeneous horizontal tube of unit length, which is made of a material modelled by the Kelvin--Voigt model for linear viscoelasticity. If we assume that the tube satisfies the Euler--Bernoulli hypothesis for the displacements and that it is subjected to no external tension, the equation governing the small transverse motions of such a tube conveying an ideal incompressible fluid is given by the mixed parabolic-hyperbolic partial differential equation (see \cite[Section 3.3]{Paidoussis2014})
\begin{equation}\label{eq001}
\frac{\partial^4\bm{w}\left(s,t\right)}{\partial s^4}+\alpha\frac{\partial^5\bm{w}\left(s,t\right)}{\partial s^4\partial t}+\eta\frac{\partial^2\bm{w}\left(s,t\right)}{\partial s^2}+2\beta \eta^{1/2}\frac{\partial^2\bm{w}\left(s,t\right)}{\partial s\partial t}+\delta\frac{\partial \bm{w}\left(s,t\right)}{\partial t}+\frac{\partial^2\bm{w}\left(s,t\right)}{\partial t^2}=0.
\end{equation}
Here $\bm{w}\left(s,t\right)$ represents the transverse displacement of the tube for $s\in\left[0,1\right]$, $t\in\mathbf{R}_+$. The parameter $\alpha>0$ is the viscoelastic damping coefficient, $\eta\geq 0$ represents the velocity of the fluid, and $\beta\in\left(0,1\right)$ is a parameter depending only on the tube and fluid densities. The remaining parameter $\delta\geq 0$ corresponds to the viscous damping due to friction from air.

The partial differential equation is supplemented with given initial displacements and velocities at time $t=0$,
\begin{equation}\label{eq002}
\bm{w}\left(s,0\right)=g\left(s\right),\qquad \left.\frac{\partial \bm{w}\left(s,t\right)}{\partial t}\right|_{t=0}=h\left(s\right),
\end{equation}
and \textit{generalised} boundary conditions at $s=0$ and $s=1$, that is,
\begin{align}
\left.\dfrac{\partial^2\bm{w}\left(s,t\right)}{\partial s^2}\right|_{s=0}&=\left. k_{01}\frac{\partial \bm{w}\left(s,t\right)}{\partial s}\right|_{s=0} +\left.k_{02}\frac{\partial^2\bm{w}\left(s,t\right)}{\partial s\partial t}\right|_{s=0},\label{eq003a}\\
-\left.\frac{\partial^3\bm{w}\left(s,t\right)}{\partial s^3}\right|_{s=0}&= \left.k_{03} \bm{w}\left(0,t\right)\right.+\left.k_{04}\frac{\partial \bm{w}\left(s,t\right)}{\partial t}\right|_{s=0},\label{eq003b}\\
-\left.\dfrac{\partial^2\bm{w}\left(s,t\right)}{\partial s^2}\right|_{s=1}&=\left. k_{11}\frac{\partial \bm{w}\left(s,t\right)}{\partial s}\right|_{s=1} +\left.k_{12}\frac{\partial^2\bm{w}\left(s,t\right)}{\partial s\partial t}\right|_{s=1},\label{eq003c}\\
\left.\frac{\partial^3\bm{w}\left(s,t\right)}{\partial s^3}\right|_{s=1}&= \left.k_{13} \bm{w}\left(1,t\right)\right.+\left.k_{14}\frac{\partial \bm{w}\left(s,t\right)}{\partial t}\right|_{s=1},\label{eq003d}
\end{align}
where the boundary parameters $k_{0j},k_{1j}\geq 0$, $j= 1, 2, 3, 4$. The system comprising the partial differential equation \eqref{eq001} together with the initial and boundary conditions \eqref{eq002}--\eqref{eq003d} forms what we call the \textit{initial/boundary-value problem}.

It is not difficult to see that from the generalised boundary conditions a variety of standard boundary conditions are obtained as special cases. For example, the boundary conditions at $s=0$ would correspond to a clamped end when $k_{01} =k_{03} = 0$ and $k_{02} , k_{04} \rightarrow\infty$, a free end when $k_{01}=k_{02} =k_{03} =k_{04} = 0$, a hinged end when $k_{01}=k_{02}=k_{03} =0$ and $k_{04}\rightarrow\infty$, or a guided end when $k_{01}=k_{03}=k_{04}=0$ and $k_{02}\rightarrow\infty$. For finite nonzero values of the $k_{0j}$, the physical interpretation of the boundary conditions is that there are viscoelastic supports at $s=0$, so that the bending moment and shear force are proportional, respectively, to the slope and the displacement, as well as their time rates of change. The boundary conditions at $s=1$ can be interpreted similarly. (We note that, strictly speaking, while for $\alpha=0$ the boundary conditions would capture the ``true'' physics of the problem, for $\alpha> 0$, as assumed here, they do not; see  \cite{BanksInman1991} or \cite{Russell1986} for explanation. But this should not impede the significance of the technical results to be presented in the paper.)

Let $\bm{w}\left(s,t\right)=e^{\lambda t}w\left(\lambda,s\right)$ so that we can separate the variables and end up with the following fourth-order ordinary differential equation:
\begin{equation}\label{301a}
(1+\alpha\lambda)\,w^{(4)}\left(\lambda,s\right)+\eta w^{(2)}\left(\lambda,s\right)+2\lambda\beta\eta^{1/2} w'\left(\lambda,s\right)+(\delta\lambda+\lambda^2)\,w\left(\lambda,s\right)=0,\\
\end{equation}
subject to the boundary conditions
\begin{alignat}{2}
&w^{(2)}\left(\lambda,0\right)-\left(k_{01}+\lambda k_{02}\right) w'\left(\lambda,0\right)&&=0,\label{302a}\\
&w^{(3)}\left(\lambda,0\right)+\left(k_{03}+\lambda k_{04}\right) w\left(\lambda,0\right)&&=0,\label{302b}\\
&w^{(2)}\left(\lambda,1\right)+\left(k_{11}+\lambda k_{12}\right) w'\left(\lambda,1\right)&&=0,\label{302c}\\
&w^{(3)}\left(\lambda,1\right)-\left(k_{13}+\lambda k_{14}\right) w\left(\lambda,1\right)&&=0.\label{302d}
\end{alignat}
Notice that the boundary conditions are all linearly dependent on the eigenvalue parameter $\lambda$ when the \textit{quadruple} $\left\{k_{02},k_{04},k_{12},k_{14}\right\}$ is nonzero. The purpose of the present paper is to give a detailed asymptotic analysis of the boundary-eigenvalue problem \eqref{301a}--\eqref{302d} for general finite (but nonnegative) values of the boundary parameters. Using these results we obtain explicit estimates and asymptotic expressions for the eigenvalues. More specifically, we are able to show for some special cases of the boundary conditions that the eigenvalues asymptotically (for sufficiently large $\left|\lambda\right|$) lie in certain sectorial sets in the complex plane. A rather interesting finding is that, even for the limiting case $k_{04}=k_{14}= 0$, the parameters $k_{02}$ and $k_{12}$ have no effect on the asymptotic behaviour of the eigenvalues -- a property useful for the computation of instability regions.

A large number of papers have appeared over the last few decades dealing with the spectral problems associated with tubes conveying fluid or related systems (such as thin panels in supersonic flow, axially moving beams, etc.), the emphasis usually being placed on studying the eigenvalues under variations in the quadruple $\left\{\alpha,\beta,\eta, \delta\right\}$. To cite but a few, these include \cite{Holmes1977,NesterovAkulenko2008,PaidoussisIssid1974,Dotsenko1979,Movchan1965,Handelman1955,BajajEtAl1980,ZhuEtAl2000} and a series of papers written from the viewpoint of operators notably by Pivovarchik \cite{Pivovarchik1993,Pivovarchik1994,Pivovarchik2005,Pivovarchik1992} and Miloslavskii \cite{Miloslavskii1985,Miloslavskii1983,Miloslavskii1981,Miloslavskii1991} (see also the closely related papers \cite{MiloslavskiiEtAl1985,Roh1982,Lyong1993}). In those papers the boundary conditions are considered typically to be clamped, free, or hinged and thus are $\lambda$-independent. It should be noted that the boundary-eigenvalue problem \eqref{301a}--\eqref{302d} is more difficult than those associated with $\lambda$-independent boundary conditions, not necessarily depending on whether $\eta\neq0$ or not. One reason is that even in the very special case $\alpha=\delta=\eta=0$ it is impossible to recast \eqref{301a}--\eqref{302d} abstractly as a spectral problem for a linear operator in the Hilbert space $\bm{L}_2\left(0,1\right)$, and one must therefore pose the problem on an appropriately extended space. The other reason is that a rather involved mathematical development will be required to derive explicit asymptotic formulae for the eigenvalues of the boundary-eigenvalue problem. Specifically, one must investigate carefully the exhibited dependencies of the eigenvalue asymptotics on the boundary parameters $\left\{k_{02},k_{04},k_{12},k_{14}\right\}$ for various parameter assumptions in $\left\{\alpha,\beta,\eta,\delta\right\}$. In the present paper, we will give a relatively simple analytical framework for dealing with these matters.

The central focus of the aforementioned works has been on the study of stability for solutions of the initial/boundary-value problems in the corresponding situations, and some interesting results on the problem of \textit{exponential} stability may be found in the papers \cite{Miloslavskii1985,Miloslavskii1983,Roh1982,MiloslavskiiEtAl1985}. More recently in the papers \cite{Khemmoudj2021,AissaEtAl2021} there has been some effort focused on energy considerations in state space terms to study the case of exponential stability problems for time-varying fluid velocities (see \cite{Artamonov2000} for an earlier example with time-varying fluid velocities). In \cite{Khemmoudj2021} the author considers the model of the tube with $\alpha=\delta=0$ and with a viscoelastic memory kernel. The tube is clamped at one end and friction exists at the other end, thus providing damping through the boundary conditions. In \cite{AissaEtAl2021} the authors analyse the exponential stability of the tube with $\alpha=0$, hinged at one end, and with damping at the other end. An important feature of the two papers is the inclusion of ``energy-conservative'' external tension in the model, which ultimately permits the investigation of exponential stability. (A similar model but with both ends hinged and without the inclusion of external tension appears in Miloslavskii's paper \cite{Miloslavskii1991}.)

In \cite{Khemmoudj2021,AissaEtAl2021}, fundamental questions concerning the location of the spectrum of the boundary-eigenvalue problem and asymptotics of eigenvalues are not considered. Results in this direction, as is well known, would offer one an alternative point of departure by using a ``spectral approach'' to studying the problem of exponential stability (refer to \cite{Miloslavskii1983,Miloslavskii1985,Roh1982} for explanation). In this vein, much of the motivation for the study of the boundary-eigenvalue problem \eqref{301a}--\eqref{302d} has come from the study of exponential stability for solutions of the initial/boundary-value problem \eqref{eq001}--\eqref{eq003d}. The need for such study is obvious in view of the integral part that \eqref{eq001} has come to play in the modelling of flow-induced oscillations for the purpose of control. The results to be obtained in the present paper can be used to fill this important gap that exists in the literature for the particular system studied.

The contents of this paper are as follows. In Section \ref{sec2} the boundary-eigenvalue problem \eqref{301a}--\eqref{302d} is defined in a suitable Hilbert product space so that it will be possible to properly discuss it from a purely operator-theoretic point of view. Some preliminary results on the structure of the spectrum and location of eigenvalues of the boundary-eigenvalue problem are also given in the section. Our main results concerning eigenvalue asymptotics are contained in Sections \ref{sec3} and \ref{sec3-2}, and in Section \ref{sec4}, we thought it useful to indicate various open problems associated with the well-posedness and stability of the initial/boundary-value problem.

\section{Abstract problem formulation and spectrum}\label{sec2}

In this section we formally obtain the abstract form of the boundary-eigenvalue problem \eqref{301a}--\eqref{302d}, and we begin by setting
\begin{equation}\label{203}
v=\left(1+\alpha\lambda\right) w
\end{equation}
where we assume that $\lambda\neq -1/\alpha$ (as a separate premise). Using \eqref{203} in \eqref{301a}, we obtain the coupled equations
\begin{align*}
\lambda w&=\frac{v-w}{\alpha},\\
\lambda v&=-\alpha v^{(4)}-2\beta\eta^{1/2} v'-\frac{\alpha\delta-1}{\alpha}v-\alpha\eta w^{(2)}+2\beta\eta^{1/2} w'+\frac{\alpha\delta-1}{\alpha}w.
\end{align*}
Let us define the new variables
\begin{alignat*}{2}
\vartheta_0&\coloneqq v'\left(0\right),\qquad\xi_0&&\coloneqq v\left(0\right),\\
\vartheta_1&\coloneqq v'\left(1\right),\qquad\xi_1&&\coloneqq v\left(1\right).
\end{alignat*}
The boundary conditions \eqref{302a}--\eqref{302d}, after using \eqref{203} together with the above definitions, and on rearranging become
\begin{alignat}{3}
\lambda k_{02}&\vartheta_0&&=&&v^{(2)}\left(0\right)-k_{01}v'\left(0\right) ,\label{304a}\\
\lambda k_{04}& \xi_0&&=-&&v^{(3)}\left(0\right)-k_{03}v\left(0\right),\label{304b}\\
\lambda k_{12} &\vartheta_1&&=-&&v^{(2)}\left(1\right)- k_{11}v'\left(1\right),\label{304c}\\
\lambda k_{14}&\xi_1&&=&&v^{(3)}\left(1\right)-k_{13}v\left(1\right).\label{304d}
\end{alignat}
For later convenience we set
\begin{equation*}
z\coloneqq\left(\begin{array}{c}
\vartheta_0\\[0.25em]
\xi_0\\[0.25em]
\vartheta_1\\[0.25em]
\xi_1
\end{array}\right)\in \mathbf{C}^4.
\end{equation*}
Unless otherwise stated, it is understood that the (nonnegative) boundary parameters $k_{0j}$, $k_{1j}$ are all nonzero and finite.

To recast \eqref{301a}--\eqref{302d} as a spectral problem for a linear operator in a suitable function space, we, as usual, introduce the Sobolev--Hilbert spaces $\bm{W}_2^k\left(0,1\right)$ of elements $w\in\bm{L}_2\left(0,1\right)$ for which $w,w',\ldots,w^{(k-1)}$ exist and are absolutely continuous on $\left[0,1\right]$, and $w^{(k)}\in\bm{L}_2\left(0,1\right)$. We then define the product space
\begin{equation}\label{heq01}
\mathbb{X}\coloneqq\bm{W}^2_2\left(0,1\right)\times \bm{L}_2\left(0,1\right)\times \mathbf{C}^4
\end{equation}
which is a Hilbert space under the inner product
\begin{equation}\label{eqin01}
\left(x,\tilde{x}\right)_\mathbb{X}\coloneqq\left(w,\tilde{w}\right)_2+\left(v,\tilde{v}\right)_0+\alpha\left(z,\tilde{z}\right),\quad x,\tilde{x}\in \mathbb{X}.
\end{equation}
If $x$ and $\tilde{x}$ are elements of $\mathbb{X}$, then $x=\left({w},
{v},{ z}\right)^\top$ and $\tilde{x}=\left(\tilde{w},
\tilde{v},\tilde{ z}\right)^\top$ (the $^\top$ denotes the vector transpose). In \eqref{eqin01} we use the inner products
\begin{align*}
\left(w,\tilde{w}\right)_2&=\int^1_0w^{(2)}\left(s\right)\overline{\tilde{w}^{(2)}\left(s\right)}\,ds
+k_{03}w\left(0\right)\overline{\tilde{w}\left(0\right)}+k_{01}w'\left(0\right)\overline{\tilde{w}'\left(0\right)}\\
&\qquad+k_{13}w\left(1\right)\overline{\tilde{w}\left(1\right)}+k_{11}w'\left(1\right)\overline{\tilde{w}'\left(1\right)},\\
\left(v,\tilde{v}\right)_0&=\int^1_0v\left(s\right)\overline{\tilde{v}\left(s\right)}\,ds,
\end{align*}
and the ``weighted'' Euclidean inner product 
\begin{equation*}
\left(z,\tilde{z}\right)=k_{02}   \vartheta_{0} \overline{\tilde{\vartheta}_{0}}+  k_{04}    \xi_{0} \overline{\tilde{\xi}_{0}}+  k_{12}   \vartheta_{1} \overline{\tilde{\vartheta}_{1}} +k_{14}   \xi_{1} \overline{\tilde{\xi}_{1}}=\left(Mz,\tilde{z}\right)_{\mathbf{C}^4}
\end{equation*}
wherein the matrix $M$ is defined as $M \coloneqq\operatorname{diag}\left(k_{02},k_{04},k_{12},k_{14}\right)$. We remark that with the choice of inner product as \eqref{eqin01}, the norm in $\mathbb{X}$ is related to the energy norm for the initial/boundary-value problem \eqref{eq001}--\eqref{eq003d} by
\begin{equation*}
E\left[\bm{w}\left(\,\cdot\,,t\right),\bm{v}\left(\,\cdot\,,t\right)\right]\coloneqq\frac{1}{2}\left\|\left(\begin{array}{c}
{\bm{w}}\left(\,\cdot\,,t\right)\\[0.25em]
\bm{v}\left(\,\cdot\,,t\right)\\[0.25em]
\bm{z}\left(t\right)
\end{array}\right)\right\|^2_\mathbb{X}=\frac{1}{2}\left\|\bm{w}\left(\,\cdot\,,t\right)\right\|^2_2+\frac{1}{2}\left\|\bm{v}\left(\,\cdot\,,t\right)\right\|^2_0+\frac{\alpha}{2}\left\|\bm{z}\left(t\right)\right\|^2,
\end{equation*}
where $\bm{w}\left(\,\cdot\,,t\right)$ is a solution of \eqref{eq001}--\eqref{eq003d}, $\bm{v}\left(\,\cdot\,,t\right)=\left(\partial \bm{w}/\partial t\right)\left(\,\cdot\,,t\right)$, and
\begin{equation*}
\bm{z}\left(t\right)=\left(\begin{array}{c}
\left.\left(\partial \bm{v}/\partial s\right)\left(s,t\right)\right|_{s=0}\\[0.25em]
\bm{v}\left(0,t\right)\\[0.25em]
\left.\left(\partial \bm{v}/\partial s\right)\left(s,t\right)\right|_{s=1}\\[0.25em]
\bm{v}\left(1,t\right)
\end{array}\right).
\end{equation*}

It is now useful to rewrite the boundary conditions \eqref{304a}--\eqref{304d} in a more appropriate form, for the purposes of this paper, by defining the boundary operators $\Gamma$ and $\Lambda$ by
\begin{equation*}\arraycolsep=1.4pt
\Gamma v= \left(\begin{array}{c}
v'\left(0\right)\\[0.25em]
v \left(0\right)\\[0.25em]
v'\left(1\right) \\[0.25em]
 v \left(1\right)
\end{array}\right),\quad\Lambda v= M^{-1}\left(\begin{array}{rcl}
v^{(2)}\left(0\right)&-&k_{01}v'\left(0\right)\\[0.25em]
-v^{(3)}\left(0\right)&-&k_{03} v \left(0\right)\\[0.25em]
-v^{(2)}\left(1\right)&-&k_{11}v'\left(1\right) \\[0.25em]
v^{(3)}\left(1\right)&-&k_{13} v \left(1\right)
\end{array}\right),\quad v\in{\bm{W}}^4_2\left(0,1\right).
\end{equation*}
The boundary conditions can then be expressed as
\begin{equation*}
\Gamma v=z,\quad\Lambda v=\lambda z.
\end{equation*}
Let us then define a new inner product in $\mathbb{X}$ by
\begin{equation}\label{211}
\left<x,\tilde{x}\right>_\mathbb{X} \coloneqq\alpha^2 \left(w,\tilde{w}\right)_2+\left(v-w, \tilde{v}-\tilde{w}\right)_0+\alpha \left(z-\Gamma w,\tilde{z}-\Gamma\tilde{w}\right),\quad x,\tilde{x}\in \mathbb{X}.
\end{equation}
This choice is motivated by the energy norm, and we see immediately that the norm $\left\|\,\cdot\,\right\|'_\mathbb{X}$ associated with the inner product $\left<\,\cdot\,,\,\cdot\,\right>_\mathbb{X}$  is (topologically) equivalent to the norm $\left\|\,\cdot\,\right\|_\mathbb{X}$ associated with $\left(\,\cdot\,,\,\cdot\,\right)_\mathbb{X}$. We will use this new inner product throughout the paper when necessary.

Define now the linear operators $A_0$ and $A_1$ in $\mathbb{X}$ by
\begin{align}
A_0{x}&=\left(\begin{array}{c}
\dfrac{v-w}{\alpha}\\[1em]
-\alpha v^{(4)}-\dfrac{\alpha\delta-1}{\alpha}\left(v-w\right)\\[1em]
\Lambda v
\end{array}
\right),\label{208xy}\\[0.5em]
A_1{x}&=\left(\begin{array}{c}
0\\[1em]
-\alpha\eta w^{(2)}-2\beta\eta^{1/2} \left(v'-w'\right)\\[1em]
0
\end{array}
\right),\label{208xyy}
\end{align}
with domains
\begin{align}
\bm{D}\left(A_0\right)&=\left\{x=\left(\begin{array}{c}
w\\[0.25em]
v\\[0.25em]
z
\end{array}\right)\in \mathbb{X}~\middle|
~\begin{gathered}
\left(\begin{array}{c}
w\\[0.25em]
v\\[0.25em]
z
\end{array}\right)\in \bm{W}^2_2\left(0,1\right)\times \bm{W}^4_2\left(0,1\right)\times \mathbf{C}^4, \\
\Gamma v=z
\end{gathered}\right\},\label{208xyz}\\[0.5em]
\bm{D}\left(A_1\right)&=\left\{x=\left(\begin{array}{c}
w\\[0.25em]
v\\[0.25em]
z
\end{array}\right)\in \mathbb{X}~\middle|
~ \left(\begin{array}{c}
w\\[0.25em]
v\\[0.25em]
z
\end{array}\right)\in \bm{W}^2_2\left(0,1\right)\times \bm{W}^1_2\left(0,1\right)\times \mathbf{C}^4\right\}.\label{208xyzz}
\end{align}
Clearly $\bm{D}\left(A_0\right)\subset\bm{D}\left(A_1\right)$, and it will be shown -- see Proposition \ref{P-3-1} below -- that $A_1$ is $A_0$-bounded but not $A_0$-compact (in the sense of \cite[Section IV.1]{Kato1995}). The boundary-eigenvalue problem \eqref{301a}--\eqref{302d} is then easily verified to be equivalent to the abstract spectral problem
\begin{equation*}
\left(\lambda I-A\right)x=0,\quad x\in \bm{D}\left({A}\right),\quad \lambda\in \mathbf{C}\backslash\left\{-1/\alpha\right\},
\end{equation*}
with $A\coloneqq A_0+A_1$, which implies that $\bm{D}\left(A\right)=\bm{D}\left(A_0\right)$. So the eigenvalues of \eqref{301a}--\eqref{302d} coincide (including multiplicities) with those of the operator $A$.

Let us recall here in brief some standard notions from the spectral theory of operators in a Hilbert space. Suppose $\lambda\mapsto \left(\lambda I-A\right)$ be a mapping from $\mathbf{C}$ into the set of closed linear operators in $\mathbb{X}$. The resolvent set of $A$, denoted by $\varrho\left(A\right)$, is the set of $\lambda$ for which $\lambda I-A$ is boundedly invertible (that is $\left(\lambda I-A\right)^{-1}$ is closed and bounded). We call this inverse $\left(\lambda I-A\right)^{-1}$ the resolvent of $A$. The spectrum $\sigma\left(A\right)$ of $A$ is the set of $\lambda$ which are not in the resolvent set. If a number $\lambda_0\in\mathbf{C}$ in the spectrum of $A$ has the property that $\operatorname{ker}\left(\lambda_0 I-A\right)\neq\left\{0\right\}$ then it is called an eigenvalue of $A$ and there exists an eigenvector $x_0\neq 0$ of $A$ corresponding to $\lambda_0$ such that $\left(\lambda_0 I-A\right)x_0=0$. The set of all eigenvalues forms the point spectrum of $A$.

We close this section with the following two propositions.
\begin{proposition}\label{P-3-1x}
Let $A_0$ be defined by \eqref{208xy}, \eqref{208xyz}. The following assertions hold:
\begin{enumerate}[\normalfont(i)]
\item\label{T-3-2-a} $0\in\varrho\left(A_0\right)$ and $A_0^{-1}$ is not compact for $\alpha\neq 0$.
\item\label{T-3-2-b} $A_0$ is maximal dissipative with respect to the inner product $\left<\,\cdot\,,\,\cdot\, \right>_\mathbb{X} $ as defined by \eqref{211}.
\item\label{T-3-2-c} $A_0$ generates a $C_0$-semigroup of bounded linear operators on $\mathbb{X}$.
\end{enumerate}
\end{proposition}
\begin{proof}
In order to establish assertion \ref{T-3-2-a}, let us consider the equation
\begin{equation*}\label{eq12xa34}
A_0x=\tilde{x}
\end{equation*}
for any $\tilde{x}\in\mathbb{X}$ and $x\in\bm{D}\left(A_0\right)$, which is equivalent to the problem posed by
\begin{align*}
\frac{v-w}{\alpha}&=\tilde{w},\\
-\alpha v^{(4)}-\frac{\alpha\delta-1}{\alpha}\left(v-w\right) &=\tilde{v},\\[0.25em]
\Lambda v&=\tilde{z},
\end{align*}
together with the ``compatibility'' conditions
\begin{equation*}
\vartheta_0=v'\left(0\right),\quad \xi_0=v\left(0\right),\quad \vartheta_1=v'\left(1\right),\quad \xi_1=v\left(1\right).
\end{equation*}
We will find it convenient to rephrase the problem in terms of
\begin{equation}
v^{(4)}=-\frac{\left(\alpha\delta-1\right)\tilde{w}+\tilde{v}}{\alpha}\coloneqq \tilde{f}\label{306xb}
\end{equation}
and the boundary conditions
\begin{alignat}{3}
 k_{02}&\tilde{\vartheta}_0&&=&&v^{(2)}\left(0\right)-k_{01}v'\left(0\right) ,\label{304xa}\\
 k_{04}& \tilde{\xi}_0&&=-&&v^{(3)}\left(0\right)-k_{03}v\left(0\right),\label{304xb}\\
 k_{12} &\tilde{\vartheta}_1&&=-&&v^{(2)}\left(1\right)- k_{11}v'\left(1\right),\label{304xc}\\
 k_{14}&\tilde{\xi}_1&&=&&v^{(3)}\left(1\right)-k_{13}v\left(1\right).\label{304xd}
\end{alignat}
It is immediate that the general solution of \eqref{306xb} takes the form
\begin{equation}\label{eqal01xc}
v\left(s\right)=v\left(0\right)+sv'\left(0\right)+\frac{s^2}{2}v^{(2)}\left(0\right)+\frac{s^3}{3!}v^{(3)}\left(0\right)+\int_0^s\frac{\left(s-r\right)^3}{3!}\tilde{f}\left(r\right)dr.
\end{equation}
Substituting \eqref{eqal01xc} in \eqref{304xa}--\eqref{304xd} yields the following system of algebraic equations for $v\left(0\right)$, $v'\left(0\right)$, $v^{(2)}\left(0\right)$, $v^{(3)}\left(0\right)$:
\begin{equation*}
\begin{split}
-k_{01}v'\left(0\right)+v^{(2)}\left(0\right)&= k_{02}\tilde{\vartheta}_0,\\
k_{03}v\left(0\right)+v^{(3)}\left(0\right)&= -k_{04}\tilde{\xi}_0,\\
k_{11}v'\left(0\right)+\bigl(1+k_{11}\bigr)\,v^{(2)}\left(0\right)+\bigl(1+\frac{k_{11}}{2}\bigr)\,v^{(3)}\left(0\right)\hspace{-0.2\linewidth}&\\
&=-k_{12}\tilde{\vartheta}_1-\int_0^1\left(1-r\right)\tilde{f}\left(r\right)dr-k_{11}\int_0^1\frac{\left(1-r\right)^2}{2}\tilde{f}\left(r\right)dr,\\
-k_{13}v\left(0\right)-k_{13}v'\left(0\right)-\frac{k_{13}}{2}v^{(2)}\left(0\right)+\bigl(1-\frac{k_{13}}{3!}\bigr)\,v^{(3)}\left(0\right)\hspace{-0.2\linewidth}&\\
&=k_{14}\tilde{\xi}_1-\int_0^1\tilde{f}\left(r\right)dr+k_{13}\int_0^1\frac{\left(1-r\right)^3}{3!}\tilde{f}\left(r\right)dr.
\end{split}
\end{equation*}
It can be shown by direct calculation that the determinant of the coefficient matrix formed by this system of equations does not vanish, that is,
\begin{equation*}
\det\,\Bigggggl(\begin{array}{cccc}
0 & -k_{01} & 1 & 0\\[0.25em]
k_{03} & 0 & 0 & 1\\[0.25em]
0 & k_{11} & 1+k_{11} & 1+\frac{k_{11}}{2}\\[0.25em]
-k_{13} & -k_{13} & -\frac{k_{13}}{2} & 1-\frac{k_{13}}{3!}
\end{array}\Bigggggr)\neq0.
\end{equation*}
So there is a solution $(v\left(0\right),v'\left(0\right), v^{(2)}\left(0\right) , v^{(3)}\left(0\right)\!{)}^\top$ which is uniquely determined by $(\tilde{w},\tilde{v},\tilde{z}{)}^\top$. Let $v\in \bm{W}^4_2\left(0,1\right)$ be a solution of \eqref{306xb}--\eqref{304xd}. Since
\begin{equation*}
w=-\alpha\tilde{w}+v,\quad z=(v\left(0\right),v'\left(0\right),v\left(1\right),
v'\left(1\right)\!{)}^\top,
\end{equation*}
it is clear that
\begin{equation*}
x=\left(\begin{array}{c}
w\\ [0.25em]
v\\ [0.25em]
z
\end{array}\right)\in \bm D\left(A_0\right),\quad A_0x=\left(\begin{array}{c}
\tilde{w}\\[0.25em]
 \tilde{v}\\[0.25em]
 \tilde{z}
\end{array}\right).
\end{equation*}
Therefore,
\begin{equation*}
A_0^{-1}\left(\begin{array}{c}
\tilde{w}\\ [0.25em]
\tilde{v}\\ [0.25em]
\tilde{z}
\end{array}\right)=\left(\begin{array}{c}
w\\[0.25em]
 v\\ [0.25em]
z
\end{array}\right) =\left(\begin{array}{c}
-\alpha \tilde{w}+v\\ [0.25em]
v\\ [0.25em]
z
\end{array}\right),
\end{equation*}
implying that $A_0$ is boundedly invertible and $A_0$ is hence closed and densely defined (as we show, $\operatorname{Re}\left<A_0x,x \right>_\mathbb{X}\le 0$). Thus $0\in\varrho\left(A_0\right)$, and it follows that $A_0^{-1}$ is not a compact operator on $\mathbb{X}$ for $\alpha\neq 0$.

In order to prove assertion \ref{T-3-2-b}, we compute the real part of the inner product $\left<A_0x,x \right>_\mathbb{X} $; a tedious calculation using integration by parts together with the fact that $\Gamma v=z$ shows that
\begin{align*}
2\operatorname{Re}\left<A_0x,x \right>_\mathbb{X} &=\left<A_0x,x \right>_\mathbb{X} +\left<x,A_0x \right>_\mathbb{X} \\
&=-2\alpha\int_0^1\bigl|v^{(2)}\left(s\right)-w^{(2)}\left(s\right)\bigr|^2ds-2\delta\int_0^1\bigl|v\left(s\right)-w\left(s\right)\bigr|^2ds\\
&\qquad -2\left(k_{04}+\alpha k_{03}\right)\bigl|v\left(0\right)-w\left(0\right)\bigr|^2-2\left(k_{02}+\alpha k_{01}\right)\bigl|v'\left(0\right)-w'\left(0\right)\bigr|^2\\
&\qquad -2\left(k_{14}+\alpha k_{13}\right)\bigl|v\left(1\right)-w\left(1\right)\bigr|^2-2\left(k_{12}+\alpha k_{11}\right)\bigl|v'\left(1\right)-w'\left(1\right)\bigr|^2.
\end{align*}
The assumptions $\alpha>0$, $\delta\geq 0$, and $k_{0j},k_{1j}\geq 0$, $j= 1, 2, 3, 4$, made in the Introduction guarantee that $\operatorname{Re}\left<A_0x,x \right>_\mathbb{X}  \leq 0$ for all $x\in \bm{D}\left(A_0\right)$, and thus $A_0$ is dissipative. The fact that it is maximal dissipative follows from the fact that it is boundedly invertible.

Assertion \ref{T-3-2-c} is an immediate consequence of the Lumer--Phillips theorem (for example, see \cite[Section 1.4]{Pazy1983}).
\end{proof}

\begin{proposition}\label{P-3-1}
Let the operator $A\coloneqq A_0+A_1$ be defined as in \eqref{208xy}--\eqref{208xyzz}. The following assertions hold:
\begin{enumerate}[\normalfont(i)]
\item\label{T-3-2-aa} The number $\lambda=-1/\alpha$ is in the continuous spectrum of $A_0$ and $A$.
\item\label{T-3-2-bb} $A_1$ is $A_0$-bounded but not $A_0$-compact, and $A_0^{-1}A_1A_0^{-1}$ is compact.
\item\label{T-3-2-cc} If $\lambda\in\varrho\left(A_0\right)\subset \mathbf{C}\backslash\left\{-1/\alpha\right\}$, then $\lambda\in\varrho\left(A\right)$; in particular, $\lambda I-A$ can be factored in the form
\begin{equation*}
\lambda I -A=\bigl[I-A_1\left(\lambda I-A_0\right)^{-1}\bigr]\left(\lambda I-A_0\right).
\end{equation*}
\item\label{T-3-2-dd} The eigenvalues of $A$ are distributed symmetrically with respect to the real axis in the complex plane. If $\eta=0$, all eigenvalues lie in the closed left half-plane, with at most finitely many of the eigenvalues lying in the right half-plane in case $\eta\neq 0$.
\end{enumerate}
\end{proposition}
\begin{proof}
Let $x\in \bm{D}\left(A_0\right)$ and consider
\begin{equation*}
\left(\lambda I-A_0\right)x=\left(\begin{array}{c}
\lambda w-\dfrac{v-w}{\alpha}\\[1em]
\lambda v+\alpha v^{(4)}+\dfrac{\alpha\delta-1}{\alpha}\left(v-w\right)\\[1em]
\lambda z-\Lambda v
\end{array}
\right).
\end{equation*}
Since the first row of $\left(\lambda I-A_0\right)x$ is the same as that of $\left(\lambda I-A\right)x$
(because the first component of $A_1x$ is zero), we clearly have, for $\lambda=-1/\alpha$,
\begin{equation*}
\operatorname{row}_1\,\bigl(-\frac{1}{\alpha} I-A_0\bigr)\,x=\operatorname{row}_1\,\bigl(-\frac{1}{\alpha} I-A\bigr)\,x=-\frac{v}{\alpha},
\end{equation*}
where $v$ is an element of $\bm{W}^4_2\left(0,1\right)$ but not of $\bm{W}^2_2\left(0,1\right)$. This proves assertion \ref{T-3-2-aa}.

In order to prove assertion \ref{T-3-2-bb}, note first that from the proof of Proposition \ref{P-3-1x}\ref{T-3-2-a} it follows that
\begin{equation*}
A_1A_0^{-1}\left(\begin{array}{c}
\tilde{w}\\[0.25em] 
\tilde{v}\\ [0.25em]
\tilde{z}
\end{array}\right) =\left(\begin{array}{c}
0\\ [1em]
\alpha^2\eta \tilde{w}^{(2)}-2\alpha\beta\eta^{1/2} \tilde{w}'-\alpha\eta{v}^{(2)}\\ [1em]
0
\end{array}\right)
\end{equation*}
and $A_1A_0^{-1}$ is a bounded linear operator on $\mathbb{X}$; but since $\tilde{w}\in\bm{W}^2_2\left(0,1\right)$ and ${v}\in\bm{W}^4_2\left(0,1\right)$, it is clear that $A_1A_0^{-1}$ is not compact. Applying $A_0^{-1}$ to what we have calculated above one readily sees that $A_0^{-1}A_1A_0^{-1}$ is compact and, moreover, $A=A_0+A_1=(I+A_1A_0^{-1})\,A_0$, which is a relatively bounded perturbation of the closed operator $A_0$ and hence is closed. Let us verify this claim. If we introduce the equivalent norm $\llangle\,\cdot\,\rrangle_\mathbb{X}$ in $\mathbb{X}$ defined by
\begin{equation*}
\llangle x\rrangle_{\mathbb{X}}\coloneqq \left[a\left(\int^1_0\,\bigl|w^{(2)}\left(s\right)\bigr|^2\,ds+\int^1_0\,\bigl|w'\left(s\right)\bigr|^2\,ds+\int^1_0\,\bigl |w\left(s\right)\bigr|^2\,ds\right)
+b\int^1_0\,\bigl|v\left(s\right)\bigr|^2\,ds+\left\|z\right\|^2_{\mathbf{C}^4}\right]^{1/2}
\end{equation*}
with $a$, $b$ positive constants, then, for any $x\in\bm{D}\left(A_0\right)$, we obtain the estimate
\begin{align*}\label{eq12xcvf45}
\llangle A_1x\rrangle^2_{\mathbb{X}}&=b\int^1_0\,\bigl|-\alpha\eta w^{(2)}\left(s\right)-2\beta\eta^{1/2}\, (v'\left(s\right)-w'\left(s\right))\bigr|^2\,ds\\
&\leq  2b\left(\alpha\eta\right)^2\int^1_0\,\bigl|w^{(2)}\left(s\right)\bigr|^2\,ds+8b\beta^2\eta\int^1_0\,\bigl|v'\left(s\right)-w'\left(s\right)\bigr|^2\,ds\\
&\leq  2b\left(\alpha\eta\right)^2\int^1_0\,\bigl|w^{(2)}\left(s\right)\bigr|^2\,ds+\frac{8b\left(\alpha\beta\right)^2\eta}{a} \,a\int^1_0\left|\frac{v'\left(s\right)-w'\left(s\right)}{\alpha}\right|^2ds.
\end{align*}
If we choose $a$ and $b$ so that
\begin{equation*}
\frac{8b\left(\alpha\beta\right)^2\eta}{a}\eqqcolon \theta<1,
\end{equation*}
if follows that
\begin{align*}
 2b\left(\alpha\eta\right)^2\int^1_0\,\bigl|w^{(2)}\left(s\right)\bigr|^2\,ds&\leq \frac{2b\left(\alpha\eta\right)^2}{a}\,\llangle x\rrangle^2_{\mathbb{X}},\\
 \frac{8b\left(\alpha\beta\right)^2\eta}{a}\, a\int^1_0\left|\frac{v'\left(s\right)-w'\left(s\right)}{\alpha}\right|^2ds&\leq \theta \,\llangle A_0x\rrangle^2_{\mathbb{X}},
\end{align*}
since
\begin{align*}
\llangle A_0x\rrangle^2_{\mathbb{X}}&= a\left(\int^1_0\,\biggl|\frac{v^{(2)}\left(s\right)-w^{(2)}\left(s\right)}{\alpha}\biggr|^2\,ds+\int^1_0\,\biggl|\frac{v'\left(s\right)-w'\left(s\right)}{\alpha}\biggr|^2\,ds
+\int^1_0\,\biggl|\frac{v\left(s\right)-w\left(s\right)}{\alpha}\biggr|^2\,ds\right)\\
&\qquad+b\int^1_0\,\biggl|-\alpha v^{(4)}\left(s\right)-\frac{\alpha\delta-1}{\alpha}\left(v\left(s\right)-w\left(s\right)\right)\biggr|^2\,ds+\left\|\Lambda v\right\|^2_{\mathbf{C}^4},
\end{align*}
and hence, that
\begin{equation*}
\llangle A_1x\rrangle_{\mathbb{X}}\leq \left(\frac{2b\left(\alpha\eta\right)^2}{a}\,\llangle x\rrangle^2_{\mathbb{X}}+\theta \,\llangle A_0x\rrangle^2_{\mathbb{X}}\right)^{{1}/{2}}
\leq \frac{\sqrt{2b}\,\alpha\eta}{\sqrt{a}}\,\llangle x\rrangle_{\mathbb{X}}+\sqrt{\theta} \,\llangle A_0x\rrangle_{\mathbb{X}}.
\end{equation*}
This proves that the $A_0$-bound is less than one and thus completes the proof that $A$ is closed (see \cite[Theorem IV.1.1]{Kato1995}). The operator $A$ therefore is a closed, densely defined linear operator in $\mathbb{X}$ with noncompact resolvent.

Thus, we have verified that $\lambda\mapsto \left(\lambda I-A\right)$ is a mapping from $\mathbf{C}$ into the set of closed linear operators in $\mathbb{X}$. Note that if $\lambda\in\varrho\left(A_0\right)$, then $\lambda\in\varrho\left(A\right)$ if and only if $I-A_1\left(\lambda I-A_0\right)^{-1}$ is boundedly invertible, because
\begin{equation*}
\left(\lambda I-A\right)=\bigl[I-A_1\left(\lambda I-A_0\right)^{-1}\bigr]\left(\lambda I-A_0\right).
\end{equation*}
Since $\left(\lambda I-A_0\right)^{-1}A_1\left(\lambda I-A_0\right)^{-1}$ is compact for each fixed $\lambda\in\varrho\left(A_0\right)$, $A_1\left(\lambda I-A_0\right)^{-1}$ is a bounded linear operator on $\mathbb{X}$ whose spectrum consists only of eigenvalues. Therefore, in $\mathbf{C}\backslash\left\{-1/\alpha\right\}$, we need only study the eigenvalues of $A$. This proves assertion \ref{T-3-2-cc}.

We now turn to the proof of assertion \ref{T-3-2-dd}. Let ${x}\left(\neq 0\right)$ be an eigenvector of $A$ corresponding to an eigenvalue ${\lambda}  $. Then
\begin{equation*}
\left(\lambda    I-A\right){x} =0,
\end{equation*}
and it follows that 
\begin{equation*}
\overline{\left(\lambda    I-A\right){x}  }=(\overline{\lambda   } I-A)\,\overline{{x}  }=0,
\end{equation*}
which means that $\overline{{x}  }$ is an eigenvector corresponding to $\overline{\lambda  }$. This proves that the spectrum of $A$ is symmetric with respect to the real axis. Now, setting $\eta=0$ so that $A_1=0$, we take the inner product of $\left(\lambda  I-A_0\right){x}$ with ${x}$ to obtain
\begin{equation*}
\left<\left(\lambda  I-A_0\right)x,x \right>_\mathbb{X} =\lambda\left\| x \right\|'^{\,2}_\mathbb{X}-\left<A_0x,x \right>_\mathbb{X} =0.
\end{equation*}
The real part of this equation is
\begin{equation*}
\operatorname{Re}{\lambda }\left\| x \right\|'^{\,2}_\mathbb{X}-\operatorname{Re}\left<A_0x,x  \right>_\mathbb{X} =0,
\end{equation*}
so, in view of Proposition \ref{P-3-1x}, we have that $\operatorname{Re}\lambda \leq 0$. Finally, consider the case where $\eta\neq 0$. We recall that for ${\lambda}  $ such that $\operatorname{Re}\lambda > 0$, the operator $A_1\left(\lambda I-A_0\right)^{-1}$ is not compact while, since $\lambda\in\varrho\left(A_0\right)$ when $\operatorname{Re}\lambda > 0$, the operator $\bigl[A_1\left(\lambda I-A_0\right)^{-1}\bigr]^2$ is; and therefore, as we have stated, $I-A_1\left(\lambda I-A_0\right)^{-1}$ is boundedly invertible if and only if $\lambda$ is not an eigenvalue of $A$. The assertion then follows from a more general perturbation result for holomorphic Fredholm operator functions (see, for example, \cite[Corollary XI.8.4]{GohbergEtAl1990}). This completes the proof of the proposition.
\end{proof}

\begin{remark}
By the well-established perturbation theory for semigroups (see \cite[Chapter 3]{Pazy1983} or \cite[Chapter IX]{Kato1995}), the operator $A$ also generates a $C_0$-semigroup, which ensures the solvability or well-posedness of the initial/boundary-value problem \eqref{eq001}--\eqref{eq003d} in $\mathbb{X}$.
\end{remark}

\section{Asymptotics of eigenvalues}\label{sec3}

We begin this section by considering the boundary-eigenvalue problem \eqref{301a}--\eqref{302d} for finite nonzero values of the $k_{0j}$, $k_{1j}$. Let us, first of all, remark that by virtue of Proposition \ref{P-3-1} the problem does not have a pure point spectrum. Since we have assumed that $\lambda\neq -1/\alpha$, we rewrite \eqref{301a} in the form
\begin{equation}\label{303}
w^{(4)}\left(\lambda,s\right)+\frac{\eta}{1+\alpha\lambda} w^{(2)}\left(\lambda,s\right)+\frac{2\lambda\beta\eta^{1/2}}{1+\alpha\lambda} w'\left(\lambda,s\right)+\frac{\delta\lambda+\lambda^2}{1+\alpha\lambda}w\left(\lambda,s\right)=0.
\end{equation}
The plane $\mathbf{C}\backslash\left\{-1/\alpha\right\}$ contains only eigenvalues, to which we will restrict ourselves in this section.

It is easy to see that asymptotically, as $\left|\lambda\right|\rightarrow\infty$, the differential equation \eqref{303} becomes
\begin{equation} \label{304}
{w}^{(4)}\left(\lambda,s\right)+\frac{2\beta\eta^{1/2}}{\alpha} {w}'\left(\lambda,s\right)+\left[\frac{\lambda}{\alpha}+\frac{1}{\alpha}\left({\delta}-\frac{1}{\alpha}\right)\right] {w}\left(\lambda,s\right)=0.
\end{equation}
In both \eqref{303} and \eqref{304} it should be noted that the equation
\begin{equation}\label{304xx}
{w}^{(4)}\left(\lambda,s\right)+\frac{\lambda}{\alpha} {w}\left(\lambda,s\right)=0
\end{equation}
represents the dominant terms and we shall use information on the eigenvalue distribution of the boundary-eigenvalue problem posed by \eqref{304xx} and the boundary conditions \eqref{302a}--\eqref{302d} to approximate the eigenvalues. This technique has its origins in the classic papers of Birkhoff \cite{Birkhoff1908a,Birkhoff1908b} and is connected with the concept of regularity of boundary-eigenvalue problems introduced in those papers. (For further details the reader should consult the book \cite{Naimark1967} by Naimark.)

The eigenvalues of \eqref{304xx}, \eqref{302a}--\eqref{302d} lie in the closed left half of the complex plane and are symmetric about the real axis (that is, those with nonzero imaginary part occur in pairs). Therefore, we need only consider the problem in the second quadrant of the complex plane, that is, we may assume that $\arg\lambda\in\left[{\pi}/{2},\pi\right]$. If we set ${\lambda}/{\alpha}=\rho^4$, then $\arg\rho\in\left[{\pi}/{8},{\pi}/{4}\right]$. To this end we define, following the technique in \cite[Section II.4.2]{Naimark1967}, the sector $\mathcal{S}$ in the complex plane by
\begin{equation*}
\mathcal{S}\coloneqq\left\{\rho\in\mathbf{C}~\middle|
~{\pi}/{8}\leq \arg\rho \leq {\pi}/{4}\right\}
\end{equation*}
and assume throughout the rest of the paper that $\rho\in\mathcal{S}$. Let us set
\begin{alignat*}{6}
\omega_1&=e^{{3\pi}i/{4}}&&=\frac{1}{\sqrt{2}}\left(i-1\right),\qquad\omega_2&&=e^{{5\pi}i/{4}}&&=-&&\frac{1}{\sqrt{2}}\left(1+i\right),\\
\omega_3&=e^{{\pi}i/{4}}&&=\frac{1}{\sqrt{2}}\left(1+i\right),\qquad \omega_4&&=e^{{7\pi}i/{4}}&&=-&&\frac{1}{\sqrt{2}}\left(i-1\right).
\end{alignat*}
Clearly $\omega_m^4+1=0$, $m=1,2,3,4$, and we have the ordering
\begin{equation*}
\operatorname{Re}\rho\omega_1\leq \operatorname{Re}\rho\omega_2\leq \operatorname{Re}\rho\omega_3 \leq \operatorname{Re}\rho\omega_4.
\end{equation*}
It is easy to see then that the inequalities $\operatorname{Re}\rho\omega_1\leq-c\left|\rho\right|$ and $\operatorname{Re}\rho\omega_2\leq 0$ always hold ($c$ being a positive constant) and asymptotically, that is for $\left|\rho\right|$ large,
\begin{equation}\label{308}
\left|e^{\rho \omega_1}\right|=\mathcal{O}\,(e^{-c\left|\rho\right|}),\qquad \left|e^{\rho \omega_2}\right|\leq1.
\end{equation}

With these preliminaries, which will be used repeatedly in the sequel, we can now produce asymptotic expressions for the eigenvalues of the boundary-eigenvalue problem \eqref{303}, \eqref{302a}--\eqref{302d}. First we prove the following lemma.
\begin{lemma}\label{L-4-1}
Let $\lambda=\alpha\rho^4$. Then, in the sector $\mathcal{S}$, the differential equation \eqref{304xx} has four linearly independent solutions $e^{\rho \omega_{m}s}$, $m=1,2,3,4$. For $\left|\rho\right|$ large, the differential equation \eqref{303} again has four linearly independent solutions $w_m\left(\rho,s\right)$, $m=1,2,3,4$, which are analytic functions of $\rho\in\mathcal{S}$ and whose asymptotic formulae are as follows:
\begin{equation}\label{eqrel023}
\left\{\begin{split}
w_{m}\left(\rho,s\right)&=e^{\rho \omega_{m}s}\,\Bigl[1+\dfrac{{W}_m\left(s\right)}{\rho}+\mathcal{O}\,(1/\rho^{2})\Bigr],\\
	w'_{m}\left(\rho,s\right)&=\rho \omega_{m}e^{\rho \omega_{m}s}\,\Bigl[1+\dfrac{{W}_m\left(s\right)}{\rho}+\mathcal{O}\,(1/\rho^{2})\Bigr],\\
w^{(2)}_{m}\left(\rho,s\right)&=\left(\rho \omega_{m}\right)^2e^{\rho \omega_{m}s}\,\Bigl[1+\dfrac{{W}_m\left(s\right)}{\rho}+\mathcal{O}\,(1/\rho^{2})\Bigr],\\
w^{(3)}_{m}\left(\rho,s\right)&=\left(\rho \omega_{m}\right)^3e^{\rho \omega_{m}s}\,\Bigl[1+\dfrac{{W}_m\left(s\right)}{\rho}+\mathcal{O}\,(1/\rho^{2})\Bigr],
\end{split}\right.
\end{equation}
where ${W}_m\left(s\right)=\frac{1}{4}\sqrt{\frac{1}{\alpha}}\left(\delta-\frac{1}{\alpha}\right)\omega_ms$.
\end{lemma}
\begin{proof}
The first statement is obvious. The proof of the second statement is along the lines of that given in \cite[Section II.4.6]{Naimark1967} and is a simple calculation based on the asymptotic expansions of fundamental solutions of the differential equation \eqref{303}. Take $\left\{w_m\left({\rho},s\right)\right\}^4_{m=1}$ to be a fundamental system of solutions to \eqref{303}, under the change from $\lambda$ to $\alpha\rho^4$ associated with $\rho\in \mathcal{S}$, having the asymptotic expansions
\begin{equation*}
w_m\left(\rho,s\right)=e^{\rho\omega_m s}\sum^{\infty}_{k=0}\frac{W_{mk}\left(s\right)}{\rho^k},\quad m=1,2,3,4.
\end{equation*}
On substituting in \eqref{303} and collecting terms according to powers of $\rho$, we calculate
\begin{equation*}
W_{m0}\left(s\right)=1,\qquad W_{m1}\left(s\right)=\frac{1}{4}\sqrt{\frac{1}{\alpha}}\left(\delta-\frac{1}{\alpha}\right)\omega_ms.
\end{equation*}
Let us write $W_{m}\left(s\right)$ for $W_{m1}\left(s\right)$ and  $W_{m2}\left(\rho,s\right)$ for $\sum^\infty_{k=2}\frac{W_{mk}\left(s\right)}{\rho^{k-2}}$. Then we see that
the functions $ W_{m2}\left(\rho,s\right)$ are uniformly bounded with respect to $s\in [0,1]$ and $\rho\in \mathcal{S}$ for $\left|\rho\right|\geq \max\left\{{1}/{\alpha}, 2\right\}$, and we have the desired result.
\end{proof}

With this lemma in place we can now determine the asymptotic expressions for the eigenvalues of  \eqref{303}, \eqref{302a}--\eqref{302d}.
\begin{theorem}\label{T-3-1}
Asymptotically, that is for large $n\in\mathbf{N}$, eigenvalues of the boundary-eigenvalue problem \eqref{303}, \eqref{302a}--\eqref{302d} are given by ${{\lambda}_n}=\alpha{\rho}_n^4$, where
\begin{equation}\label{eqsols3xa}
{\rho}_n=\left(n+{1}/{2}\right)\pi \omega_2-i\left[\frac{1}{\alpha}\left(\frac{1}{ k_{04}}+\frac{1}{ k_{14}}\right)\omega_1+\frac{1}{4}\sqrt{\frac{1}{\alpha}}\left(\delta-\frac{1}{\alpha}\right)\omega_2\right]\frac{1}{\left(n+{1}/{2}\right)\pi}+\mathcal{O}\,(1/n^{2}).
\end{equation}
\end{theorem}
\begin{proof}
Let $w\left(\rho,s\right)$ be an eigenfunction of \eqref{303}, \eqref{302a}--\eqref{302d} under the change from $\lambda$ to $\alpha\rho^4$ associated with $\rho\in \mathcal{S}$. By Lemma \ref{L-4-1},
\begin{equation}\label{eqeig01x}
w\left(\rho,s\right)=a_1w_1\left(\rho,s\right)+a_2w_2\left(\rho,s\right)+a_3w_3\left(\rho,s\right)+a_4w_4\left(\rho,s\right)
\end{equation}
for constants $a_m$, $m=1,2,3,4$, possibly dependent of $\rho$. If \eqref{eqeig01x} is inserted into the boundary conditions, then we arrive at the following system of algebraic equations:
\begin{alignat*}{2}
&\bigl[w_{1}^{(2)}\left(\rho,0\right)-(k_{11}+\alpha\rho^4 k_{12})\,w'_{1}\left(\rho,0\right)\bigr]\,a_1+\bigl[w_{2}^{(2)}\left(\rho,0\right)-(k_{01}+\alpha\rho^4 k_{02})\,w'_{2}\left(\rho,0\right)\bigr]\, a_2\\
&\quad+\bigl[w_{3}^{(2)}\left(\rho,0\right)-(k_{01}+\alpha\rho^4 k_{02})\,w'_{3}\left(\rho,0\right)\bigr]\, a_3+\bigl[w_{4}^{(2)}\left(\rho,0\right)-(k_{11}+\alpha\rho^4 k_{12})\,w'_{4}\left(\rho,0\right)\bigr]\, a_4&&=0,\\[0.25em]
&\bigl[w_{1}^{(3)}\left(\rho,0\right)+(k_{11}+\alpha\rho^4 k_{12})\,w_{1}\left(\rho,0\right)\bigr]\, a_1+\bigl[w_{2}^{(3)}\left(\rho,0\right)+(k_{11}+\alpha\rho^4 k_{12})\,w_{2}\left(\rho,0\right)\bigr]\, a_2\\
&\quad+\bigl[w_{3}^{(3)}\left(\rho,0\right)+(k_{11}+\alpha\rho^4 k_{12})\,w_{3}\left(\rho,0\right)\bigr]\, a_3+\bigl[w_{4}^{(3)}\left(\rho,0\right)+(k_{11}+\alpha\rho^4 k_{12})\,w_{4}\left(\rho,0\right)\bigr]\, a_4&&=0,\\[0.25em]
&\bigl[w_{1}^{(2)}\left(\rho,1\right)+(k_{11}+\alpha\rho^4 k_{12})\,w'_{1}\left(\rho,1\right)\bigr]\,a_1+\bigl[w_{2}^{(2)}\left(\rho,1\right)+(k_{11}+\alpha\rho^4 k_{12})\,w'_{2}\left(\rho,1\right)\bigr]\, a_2\\
&\quad+\bigl[w_{3}^{(2)}\left(\rho,1\right)+(k_{11}+\alpha\rho^4 k_{12})\,w'_{3}\left(\rho,1\right)\bigr]\, a_3+\bigl[w_{4}^{(2)}\left(\rho,1\right)+(k_{11}+\alpha\rho^4 k_{12})\,w'_{4}\left(\rho,1\right)\bigr]\, a_4&&=0,\\[0.25em]
&\bigl[w_{1}^{(3)}\left(\rho,1\right)-(k_{11}+\alpha\rho^4 k_{12})\,w_{1}\left(\rho,1\right)\bigr]\, a_1+\bigl[w_{2}^{(3)}\left(\rho,1\right)-(k_{11}+\alpha\rho^4 k_{12})\,w_{2}\left(\rho,1\right)\bigr]\, a_2\\
&\quad+\bigl[w_{3}^{(3)}\left(\rho,1\right)-(k_{11}+\alpha\rho^4 k_{12})\,w_{3}\left(\rho,1\right)\bigr]\, a_3+\bigl[w_{4}^{(3)}\left(\rho,1\right)+(k_{13}-\alpha\rho^4 k_{14})\,w_{4}\left(\rho,1\right)\bigr]\, a_4&&=0.
\end{alignat*}
Using the boundary values of the $w_{m}\left(\rho,s\right)$ and their derivatives (as defined in \eqref{eqrel023}), the boundary conditions \eqref{302a}--\eqref{302d} are seen to imply, for large $\left|\rho\right|$,
\begin{equation}\label{314}
\begin{alignedat}{1}
w_{m}^{(2)}\left(\rho,0\right)-(k_{01}+\alpha\rho^4  k_{02})\,w'_{m}\left(\rho,0\right)&=\left(\rho \omega_m\right)^2\left[1\right]_{m0}-(k_{01}+\alpha\rho^4k_{02})\,\rho \omega_m\left[1\right]_{m0},\\
w_{m}^{(3)}\left(\rho,0\right)+(k_{03}+\alpha\rho^4  k_{04})\,w_{m}\left(\rho,0\right)&=\left(\rho \omega_m\right)^3\left[1\right]_{m0}+(k_{03}+\alpha\rho^4k_{04})\, \left[1\right]_{m0},\\
w_{m}^{(2)}\left(\rho,1\right)+(k_{11}+\alpha\rho^4  k_{12})\,w'_{m}\left(\rho,1\right)&=\left(\rho \omega_m\right)^2e^{\rho \omega_m}\left[1\right]_{m1}+(k_{11}+\alpha\rho^4k_{12})\,\rho \omega_me^{\rho \omega_m}\left[1\right]_{m1},\\
w_{m}^{(3)}\left(\rho,1\right)-(k_{13}+\alpha\rho^4  k_{14})\,w_{m}\left(\rho,1\right)&=\left(\rho \omega_m\right)^3e^{\rho \omega_m}\left[1\right]_{m1}-(k_{13}+\alpha\rho^4k_{14})\,e^{\rho \omega_m} \left[1\right]_{m1},
\end{alignedat}
\end{equation}
where we have used Birkhoff's notation  $\left[1\right]_{ms}\coloneqq1+\frac{{W}_m\left(s\right)}{\rho}+\mathcal{O}\,(1/\rho^{2})$ for $m=1,2,3,4$ and $s\in\left[0,1\right]$. By combining these with the above system of algebraic equations, we have that $w\left(\rho,s\right)$ as given by \eqref{eqeig01x} is nontrivial if and only if
\begin{align*}
\det\,&\Bigggggl(\begin{matrix}
-\alpha\rho^4k_{02}\rho \omega_1\left[1\right]_{10}&-\alpha\rho^4k_{02}\rho \omega_2\left[1\right]_{20}\\[0.5em]
\bigl[\left(\rho \omega_1\right)^3+\alpha\rho^4k_{04}\bigr]\left[1\right]_{10}&\bigl[\left(\rho \omega_2\right)^3+\alpha\rho^4k_{04}\bigr]\left[1\right]_{20}\\[0.5em]
\alpha\rho^4k_{12}\rho \omega_1e^{\rho \omega_1}\left[1\right]_{11}&\alpha\rho^4k_{12}\rho \omega_2e^{\rho \omega_2}\left[1\right]_{21}\\[0.5em]
\bigl[\left(\rho \omega_1\right)^3-\alpha\rho^4k_{14}\bigr]\,e^{\rho \omega_1} \left[1\right]_{11}&\bigl[\left(\rho \omega_2\right)^3-\alpha\rho^4k_{14}\bigr]\,e^{\rho \omega_2} \left[1\right]_{21}
\end{matrix}\\[0.5em]
&\hspace{0.25\linewidth}
\begin{matrix}
-\alpha\rho^4k_{02}\rho \omega_3\left[1\right]_{30}&-\alpha\rho^4k_{02}\rho \omega_4\left[1\right]_{40}\\[0.5em]
\bigl[\left(\rho \omega_3\right)^3+\alpha\rho^4k_{04}\bigr]\left[1\right]_{30}&\bigl[\left(\rho \omega_4\right)^3+\alpha\rho^4k_{04}\bigr]\left[1\right]_{40}\\[0.5em]
\alpha\rho^4k_{12}\rho \omega_3e^{\rho \omega_3}\left[1\right]_{31}&\alpha\rho^4k_{12}\rho \omega_4e^{\rho \omega_4}\left[1\right]_{41}\\[0.5em]
\bigl[\left(\rho \omega_3\right)^3-\alpha\rho^4k_{14}\bigr]\,e^{\rho \omega_3} \left[1\right]_{31}&\bigl[\left(\rho \omega_4\right)^3-\alpha\rho^4k_{14}\bigr]\,e^{\rho \omega_4} \left[1\right]_{41}
\end{matrix}\Bigggggr)=0.
\end{align*}
Using \eqref{308} we have
\begin{equation*}
	\det\,\Bigggggl(\begin{matrix}
		 \omega_1\left[1\right]_{10}&  \omega_2\left[1\right]_{20} & -\omega_2e^{\rho \omega_2}\left[1\right]_{30}& 0\\[0.5em]
		\bigl(\frac{ \omega_1^3}{\alpha\rho k_{04}}+1\bigr)\left[1\right]_{10}&\bigl(\frac{ \omega_2^3}{\alpha\rho k_{04}}+1\bigr)\left[1\right]_{20} & 	\bigl(1-\frac{ \omega_2^3}{\alpha\rho k_{04}}\bigr)\,e^{\rho \omega_2}\left[1\right]_{30}&0\\[0.5em]
		0&\omega_2e^{\rho \omega_2}\left[1\right]_{21} & 		-\omega_2 \left[1\right]_{31}& -\omega_1\left[1\right]_{41}\\[0.5em]
		0&\bigl(1-\frac{\omega_2^3}{\alpha\rho k_{14}}\bigr)\,e^{\rho \omega_2} \left[1\right]_{21} & 		\bigl(\frac{ \omega_2^3}{\alpha\rho k_{14}}+1\bigr) \left[1\right]_{31}&\bigl(\frac{ \omega_1^3}{\alpha\rho k_{14}}+1\bigr)\left[1\right]_{41}
	\end{matrix}\Bigggggr)=0,
\end{equation*}
since $\omega_3=-\omega_2$ and $\omega_4=-\omega_1$. Then, noting that  $\left[1\right]_{m0}=1+\mathcal{O}\,(1/\rho^{2})$ and $\left[1\right]_{m1}=1+\frac{{W}_m\left(1\right)}{\rho}+\mathcal{O}\,(1/\rho^{2})$, we see that
\begin{align*}
	&\det\,\Bigggggl(\begin{matrix}
		 \omega_1& \omega_2\\[0.5em]
\frac{ \omega_1^3}{\alpha\rho k_{04}}+1&\frac{ \omega_2^3}{\alpha\rho k_{04}}+1\\[0.5em]
 0&\omega_2e^{\rho \omega_2}\,\bigl(1+\frac{W_2\left(1\right)}{\rho}\bigr)\\[0.5em]
0&\bigl(1-\frac{\omega_2^3}{\alpha\rho k_{14}}\bigr)\,e^{\rho \omega_2}\, \bigl(1+\frac{W_2\left(1\right)}{\rho}\bigr)
	\end{matrix}\\[0.5em]
	&\hspace{0.25\linewidth}
	\begin{matrix}
		  -\omega_2e^{\rho \omega_2}& 0\\[0.5em]
	\bigl(1-\frac{ \omega_2^3}{\alpha\rho k_{04}}\bigr)\,e^{\rho \omega_2}&0\\[0.5em]
  -\omega_2\,\bigl(1+\frac{W_3\left(1\right)}{\rho}\bigr)& -\omega_1\,\bigl(1+\frac{W_4\left(1\right)}{\rho}\bigr)\\[0.5em]
\bigl(\frac{ \omega_2^3}{\alpha\rho k_{14}}+1\bigr)\, \bigl(1+\frac{W_3\left(1\right)}{\rho}\bigr)&\bigl(\frac{ \omega_1^3}{\alpha\rho k_{14}}+1\bigr)\,\bigl(1+\frac{W_4\left(1\right)}{\rho}\bigr)
	\end{matrix}\Bigggggr)+\mathcal{O}\,(1/\rho^{2})=0,
\end{align*}
and a direct calculation yields
\begin{equation}\label{326}
e^{2\rho \omega_2}=-1+2\left[\frac{1}{\alpha}\left(\frac{1}{  k_{04}}+\frac{1}{  k_{14}}\right)\omega_1+\frac{1}{4}\sqrt{\frac{1}{\alpha}}\left(\delta-\frac{1}{\alpha}\right)\omega_2\right]\frac{1}{\rho}+\mathcal{O}\,(1/\rho^{2}),
\end{equation}
where we have used that $\omega_1-\omega_2=i\sqrt{2}$, $\omega_1+\omega_2=-\sqrt{2}$, and $ \omega_1\omega_2=1$. Set
\begin{equation}\label{eq1bvs}
\rho_n=\left(n+{1}/{2}\right)\pi\omega_2+\gamma_n,\quad n\in\mathbf{N}.
\end{equation}
Substituting $\rho_n$ for $\rho$ in \eqref{326} we obtain that
\begin{equation*}
2\omega_2\gamma_n=2\left[\frac{1}{\alpha}\left(\frac{1}{  k_{04}}+\frac{1}{  k_{14}}\right)\omega_1+\frac{1}{4}\sqrt{\frac{1}{\alpha}}\left(\delta-\frac{1}{\alpha}\right)\omega_2\right]\frac{1}{\left(n+{1}/{2}\right)\pi\omega_2}+\mathcal{O}\,(1/n^{2})
\end{equation*}
and thus, with $\omega_2^2=i$,
\begin{equation}\label{eq1bvsss}
\gamma_n=-i\left[\frac{1}{\alpha}\left(\frac{1}{  k_{04}}+\frac{1}{  k_{14}}\right)\omega_1+\frac{1}{4}\sqrt{\frac{1}{\alpha}}\left(\delta-\frac{1}{\alpha}\right)\omega_2\right]\frac{1}{\left(n+{1}/{2}\right)\pi}+\mathcal{O}\,(1/n^{2}).
\end{equation}
Combining \eqref{eq1bvsss} with \eqref{eq1bvs} yields \eqref{eqsols3xa}, completing the proof.
\end{proof}

\begin{remark}
Using \eqref{eqsols3xa} we can readily calculate
\begin{align*}
{\rho}^4_n&=-\Bigl[\left(n+{1}/{2}\right)\pi\Bigr]^4+4\left[\frac{1}{\alpha}\left(\frac{1}{ k_{04}}+\frac{1}{ k_{14}}\right)\right]\Bigl[\left(n+{1}/{2}\right)\pi\Bigr]^2\\
&\qquad+i\left\{4\left[\frac{1}{4}\sqrt{\frac{1}{\alpha}}\left(\delta-\frac{1}{\alpha}\right)\right]\Bigl[\left(n+{1}/{2}\right)\pi\Bigr]^2\right\}+\mathcal{O}\,(n).
\end{align*}
\end{remark}

\begin{remark}\label{Rem-3-1}
It is important to notice that the asymptotic expressions in Theorem \ref{T-3-1} are independent of the values of $\beta$ and $\eta$. In particular this means that the theorem is still valid if $\eta=0$ (that is, when there is no flow). Notably, the boundary parameters $k_{02}$ and $k_{12}$ also play essentially no role, in the asymptotic sense, as they are at most within $\mathcal{O}\,(n)$. We shall return to this matter in Section \ref{sec3-2}.
\end{remark}

\subsection{Eigenvalues for clamped boundary conditions}\label{sec3-1}

In this subsection, we will consider the boundary-eigenvalue problem corresponding to the case in which $k_{0j},k_{1j}\rightarrow\infty$ for $j=2,4$, that is, the case in which the ends of the tube are clamped. The boundary-eigenvalue problem is posed by the same differential equation \eqref{303}, which we rewrite here with the eigenvalue parameter $\lambda$ replaced by $\hat{\lambda}$:
\begin{equation}\label{403}
	w^{(4)}\,(\hat{\lambda},s)+\frac{\eta}{1+\alpha\hat{\lambda}} w^{(2)}\,(\hat{\lambda},s)+\frac{2\hat{\lambda}\beta\eta^{1/2}}{1+\alpha\hat{\lambda}} w'\,(\hat{\lambda},s)+\frac{\delta\hat{\lambda}+\hat{\lambda}^2}{1+\alpha\hat{\lambda}}w\,(\hat{\lambda},s)=0,
\end{equation}
but with (much simpler) boundary conditions
\begin{align}
	w\,(\hat{\lambda},0)&=0,\label{402a}\\
	w'\,(\hat{\lambda},0)&=0,\label{402b}\\
	w\,(\hat{\lambda},1)&=0,\label{402c}\\
	w'\,(\hat{\lambda},1)&=0.\label{402d}
\end{align}

As in the previous case of generalised boundary conditions, we wish to produce asymptotic formulae for the eigenvalues $\hat{\lambda}_n$ of the boundary-eigenvalue problem \eqref{403}--\eqref{402d}. We are particularly concerned with the relationship between these and the eigenvalues $\lambda_n$ of the boundary-eigenvalue problem \eqref{303}, \eqref{302a}--\eqref{302d}. It is expected that for $n$ large enough the eigenvalues $\lambda_n$ should tend to the eigenvalues $\hat{\lambda}_n$ asymptotically as $k_{0j},k_{1j}\rightarrow\infty$ for $j=2,4$. This property would mean that the eigenvalues depend continuously on the boundary parameters.

Let us set ${\hat{\lambda}}/{\alpha}=\hat{\rho}^4$ and assume $\hat{\rho}\in \mathcal{S}$. For \eqref{403} with ${\hat{\lambda}}={\alpha}\hat{\rho}^4$ essentially the result of Lemma \ref{L-4-1} remains valid. So proceeding as in the proof of Theorem \ref{T-3-1} but using the boundary conditions \eqref{402a}--\eqref{402d} in place of  \eqref{302a}--\eqref{302d}, the determinantal condition here reads
\begin{equation*}
\det\,\Bigggggl(\begin{array}{cccc}
\left[1\right]_{10}&\left[1\right]_{20}&\left[1\right]_{30}&\left[1\right]_{40}\\[0.5em]
\hat{\rho} \omega_1\left[1\right]_{10}&\hat{\rho} \omega_2\left[1\right]_{20}&\hat{\rho} \omega_3\left[1\right]_{30}&\hat{\rho} \omega_4\left[1\right]_{40}\\[0.5em]
e^{\hat{\rho} \omega_1} \left[1\right]_{11}&e^{\hat{\rho} \omega_2} \left[1\right]_{21}&e^{\hat{\rho} \omega_3} \left[1\right]_{31}&e^{\hat{\rho} \omega_4} \left[1\right]_{41}\\[0.5em]
\hat{\rho} \omega_1e^{\hat{\rho} \omega_1}\left[1\right]_{11}&\hat{\rho} \omega_2e^{\hat{\rho} \omega_2}\left[1\right]_{21}&\hat{\rho} \omega_3e^{\hat{\rho} \omega_3}\left[1\right]_{31}&\hat{\rho} \omega_4e^{\hat{\rho} \omega_4}\left[1\right]_{41}
\end{array}\Bigggggr)=0.
\end{equation*}
Noting that $\omega_3=-\omega_2$ and $\omega_4=-\omega_1$, the equation is equivalent to
\begin{equation*}
\det\,\Bigggggl(\begin{array}{cccc}
\left[1\right]_{10}&\left[1\right]_{20}&e^{\hat{\rho} \omega_2}\left[1\right]_{30}&e^{\hat{\rho} \omega_1}\left[1\right]_{40}\\[0.5em]
 \omega_1\left[1\right]_{10}& \omega_2\left[1\right]_{20}& -\omega_2e^{\hat{\rho} \omega_2}\left[1\right]_{30}& -\omega_1e^{\hat{\rho} \omega_1}\left[1\right]_{40}\\[0.5em]
e^{\hat{\rho} \omega_1} \left[1\right]_{11}&e^{\hat{\rho} \omega_2}\left[1\right]_{21}& \left[1\right]_{31}&\left[1\right]_{41}\\[0.5em]
 \omega_1e^{\hat{\rho} \omega_1}\left[1\right]_{11}&\omega_2e^{\hat{\rho} \omega_2}\left[1\right]_{21}& -\omega_2\left[1\right]_{31}& -\omega_1\left[1\right]_{41}
\end{array}\Bigggggr)=0
\end{equation*}
which becomes, in view of \eqref{308},
 \begin{equation*}
\det\,\Bigggggl(\begin{array}{cccc}
1&1&e^{\hat{\rho} \omega_2}&0\\[0.5em]
 \omega_1& \omega_2& -\omega_2e^{\hat{\rho} \omega_2}& 0\\[0.5em]
0&e^{\hat{\rho} \omega_2}\,\bigl(1+\frac{W_2\left(1\right)}{\hat{\rho}}\bigr)&1+\frac{W_3\left(1\right)}{\hat{\rho}}&1+\frac{W_4\left(1\right)}{\hat{\rho}}\\[0.5em]
 0&\omega_2e^{\hat{\rho} \omega_2}\,\bigl(1+\frac{W_2\left(1\right)}{\hat{\rho}}\bigr)& -\omega_2\,\bigl(1+\frac{W_3\left(1\right)}{\hat{\rho}}\bigr)&-\omega_1\,\bigl(1+\frac{W_4\left(1\right)}{\hat{\rho}}\bigr)
\end{array}\Bigggggr)+\mathcal{O}\,(1/\hat{\rho}^{2})=0,
\end{equation*}
whence
\begin{equation}\label{406}
e^{2\hat{\rho} \omega_2}=-1+2\left[\frac{1}{4}\sqrt{\frac{1}{\alpha}}\left(\delta-\frac{1}{\alpha}\right)\omega_2\right]\frac{1}{\hat{\rho}}+\mathcal{O}\,(1/\hat{\rho}^{2}),
\end{equation}
where we have again used that $\left[1\right]_{m0}=1+\mathcal{O}\,(1/\hat{\rho}^{2})$ and $\left[1\right]_{m1}=1+\frac{{W}_m\left(1\right)}{\hat{\rho}}+\mathcal{O}\,(1/\hat{\rho}^{2})$, and that $\omega_1-\omega_2=i\sqrt{2}$ and $\omega_1+\omega_2=-\sqrt{2}$. Setting
\begin{equation*}
\hat{\rho}_n=\left(n+{1}/{2}\right)\pi\omega_2+\hat{\gamma}_n,\quad n\in\mathbf{N},
\end{equation*}
we calculate
\begin{equation*}
\hat{\gamma}_n=-i\left[\frac{1}{4}\sqrt{\frac{1}{\alpha}}\left(\delta-\frac{1}{\alpha}\right)\omega_2\right]\frac{1}{\left(n+{1}/{2}\right)\pi}+\mathcal{O}\,(1/n^{2}).
\end{equation*}
The solutions of \eqref{406} are thus
\begin{equation*}
{\hat{\rho}}_n=\left(n+{1}/{2}\right)\pi\omega_2-i\left[\frac{1}{4}\sqrt{\frac{1}{\alpha}}\left(\delta-\frac{1}{\alpha}\right)\omega_2\right]\frac{1}{\left(n+{1}/{2}\right)\pi}+\mathcal{O}\,(1/n^{2})
\end{equation*}
and are precisely the solutions \eqref{eqsols3xa} of \eqref{326} when $k_{04},k_{14}\rightarrow\infty$. From this we see that the dominant terms in the ${\hat{\rho}}_n$ are the same as those in the ${\rho}_n$, when $k_{04}$ and $k_{14}$ are finite. We summarise in the following result.
\begin{theorem}\label{T-4-1}
For large $n\in\mathbf{N}$ eigenvalues of \eqref{303}, \eqref{302a}--\eqref{302d} are given as in Theorem \ref{T-3-1}. If we take the limit as $k_{04},k_{14}\rightarrow\infty$, these eigenvalues tend to those of \eqref{403}--\eqref{402d} in the asymptotic sense.
\end{theorem}

\section{Asymptotics of eigenvalues for $k_{04}=k_{14}= 0$}\label{sec3-2}

As pointed out in Remark \ref{Rem-3-1} following the proof of Theorem \ref{T-3-1}, the boundary parameters $k_{02}$ and $k_{12}$ do not appear in the asymptotic formulae for the eigenvalues $\lambda_n$ of \eqref{303}, \eqref{302a}--\eqref{302d} when $k_{04}$, $k_{14}$ are nonzero. The final problem in this paper consists of the question as to whether or not the eigenvalues depend, asymptotically, on $k_{02}$ and $k_{12}$ for the limiting case $k_{04}=k_{14}= 0$. The boundary-eigenvalue problem which we consider, then, is posed by \eqref{303} and the following boundary conditions:
\begin{alignat}{2}
w^{(2)}\left(\lambda,0\right)-\left(k_{01}+\lambda k_{02}\right) w'\left(\lambda,0\right)&=0,\label{502a}\\
w^{(3)}\left(\lambda,0\right)+k_{03} w\left(\lambda,0\right)&=0,\label{502b}\\
w^{(2)}\left(\lambda,1\right)+\left(k_{11}+\lambda k_{12}\right) w'\left(\lambda,1\right)&=0,\label{502c}\\
w^{(3)}\left(\lambda,1\right)-k_{13} w\left(\lambda,1\right)&=0.\label{502d}
\end{alignat}

The following theorem is an analogue of Theorem \ref{T-3-1} for the problem \eqref{303}, \eqref{502a}--\eqref{502d} and confirms that the eigenvalues $\lambda_n$ do not depend asymptotically on the boundary parameters $k_{02}$ and $k_{12}$.
\begin{theorem}\label{T-4-4}
For large $n\in\mathbf{N}$ eigenvalues of the boundary-eigenvalue problem \eqref{303}, \eqref{502a}--\eqref{502d} are given by ${{\lambda}_n}=\alpha{\rho}_n^4$, where
\begin{equation}\label{eq23ffgh}
{\rho}_n=n\pi\omega_2+i\left[\frac{1}{4}\sqrt{\frac{1}{\alpha}}\left(\delta-\frac{1}{\alpha}\right)\omega_2\right]\frac{1}{n\pi}+\mathcal{O}\,(1/n^{2}).
\end{equation}
\end{theorem}
\begin{proof}
Defining $w\left(\rho,s\right)$ as in the proof of Theorem \ref{T-3-1} (in \eqref{eqeig01x}), we have
\begin{equation*}
\begin{alignedat}{1}
w_{m}^{(2)}\left(\rho,0\right)-(k_{01}+\alpha\rho^4  k_{02})\,w'_{m}\left(\rho,0\right)&=\left(\rho \omega_m\right)^2\left[1\right]_{m0}-(k_{01}+\alpha\rho^4k_{02})\,\rho \omega_m\left[1\right]_{m0},\\
w_{m}^{(3)}\left(\rho,0\right)+k_{03}w_{m}\left(\rho,0\right)&=\left(\rho \omega_m\right)^3\left[1\right]_{m0}+k_{03} \left[1\right]_{m0},\\
w_{m}^{(2)}\left(\rho,1\right)+(k_{11}+\alpha\rho^4  k_{12})\,w'_{m}\left(\rho,1\right)&=\left(\rho \omega_m\right)^2e^{\rho \omega_m}\left[1\right]_{m1}+(k_{11}+\alpha\rho^4k_{12})\,\rho \omega_me^{\rho \omega_m}\left[1\right]_{m1},\\
w_{m}^{(3)}\left(\rho,1\right)-k_{13}w_{m}\left(\rho,1\right)&=\left(\rho \omega_m\right)^3e^{\rho \omega_m}\left[1\right]_{m1}-k_{13}e^{\rho \omega_m} \left[1\right]_{m1},
\end{alignedat}
\end{equation*}
which is just \eqref{314} with $k_{04}=k_{14}= 0$. So using exactly the same reasoning as in the proof of Theorem \ref{T-3-1}, we have that $w\left(\rho,s\right)$ is nontrivial if and only if 
\begin{equation*}
\det\,\Bigggggl(\begin{matrix}
-\alpha\rho^4k_{02}\rho \omega_1\left[1\right]_{10}&-\alpha\rho^4k_{02}\rho \omega_2\left[1\right]_{20} & -\alpha\rho^4k_{02}\rho \omega_3\left[1\right]_{30}&-\alpha\rho^4k_{02}\rho \omega_4\left[1\right]_{40}\\[0.5em]
\left(\rho \omega_1\right)^3\left[1\right]_{10}&\left(\rho \omega_2\right)^3\left[1\right]_{20} & \left(\rho \omega_3\right)^3\left[1\right]_{30}&\left(\rho \omega_4\right)^3\left[1\right]_{40}\\[0.5em]
\alpha\rho^4k_{12}\rho \omega_1e^{\rho \omega_1}\left[1\right]_{11}&\alpha\rho^4k_{12}\rho \omega_2e^{\rho \omega_2}\left[1\right]_{21} & \alpha\rho^4k_{12}\rho \omega_3e^{\rho \omega_3}\left[1\right]_{31}&\alpha\rho^4k_{12}\rho \omega_4e^{\rho \omega_4}\left[1\right]_{41}\\[0.5em]
\left(\rho \omega_1\right)^3e^{\rho \omega_1} \left[1\right]_{11}&\left(\rho \omega_2\right)^3e^{\rho \omega_2} \left[1\right]_{21} & \left(\rho \omega_3\right)^3e^{\rho \omega_3} \left[1\right]_{31}&\left(\rho \omega_4\right)^3e^{\rho \omega_4} \left[1\right]_{41} 
\end{matrix}\Bigggggr)=0.
\end{equation*}
Recalling that $\omega_3=-\omega_2$ and $\omega_4=-\omega_1$, and using \eqref{308} it follows immediately that
\begin{equation*}
\det\,\Bigggggl(\begin{array}{cccc}
 1& 1& e^{\rho \omega_2}& 0\\[0.5em]
\omega_1^2&\omega_2^2&\omega_2^2e^{\rho \omega_2}&0\\[0.5em]
0& e^{\rho \omega_2}\,\bigl(1+\frac{W_2\left(1\right)}{\rho}\bigr)& 1+\frac{W_3\left(1\right)}{\rho}& 1+\frac{W_4\left(1\right)}{\rho}\\[0.5em]
0&\omega_2^2e^{\rho \omega_2}\,\bigl(1+\frac{W_2\left(1\right)}{\rho}\bigr)&\omega_2^2\,\bigl(1+\frac{W_3\left(1\right)}{\rho}\bigr)&\omega_1^2\,\bigl(1+\frac{W_4\left(1\right)}{\rho}\bigr)
\end{array}\Bigggggr)+\mathcal{O}\,(1/\rho^{2})=0,
\end{equation*}
and a direct calculation gives
\begin{equation*}
e^{2\rho \omega_3}=1-2\left[\frac{1}{4}\sqrt{\frac{1}{\alpha}}\left(\delta-\frac{1}{\alpha}\right)\omega_2\right]\frac{1}{\rho}+\mathcal{O}\,(1/\rho^{2}).
\end{equation*}
Set
\begin{equation*}
\rho_n=n\pi\omega_2+\gamma_n,\quad n\in\mathbf{N}.
\end{equation*}
Using the procedure leading to Theorem \ref{T-3-1}, we arrive at
\begin{equation*}
{\gamma}_n=i\left[\frac{1}{4}\sqrt{\frac{1}{\alpha}}\left(\delta-\frac{1}{\alpha}\right)\omega_2\right]\frac{1}{n\pi}+\mathcal{O}\,(1/n^{2});
\end{equation*}
so the $\rho_n$ are as given in \eqref{eq23ffgh} and the theorem is proven.
\end{proof}

\section{Further remarks and open problems}\label{sec4}

In this paper we have been primarily concerned with studying the boundary-eigenvalue problem \eqref{301a}--\eqref{302d} and the asymptotics of eigenvalues in the plane $\mathbf{C}\backslash\left\{-1/\alpha\right\}$. What we have effectively shown in Sections \ref{sec3} and \ref{sec3-2} is that if $n$ is large enough in the sequence $\left\{\lambda_n\right\}$, then $\left|\operatorname{Im}\lambda_n\right|/\left|\operatorname{Re}\lambda_n\right|\leq 1$ and hence the eigenvalues lie in the sectorial set 
\begin{equation*}
\left\{ z\in \mathbf{C}\backslash\left\{-1/\alpha\right\}~\middle|~ \left|\operatorname{Im}\left(z-z_0\right)\right|\leq r\left|\operatorname{Re}\left(z-z_0\right)\right|,~ \operatorname{Re}\left(z-z_0\right)\leq 0\right\},
\end{equation*}
with $r$ and $z_0$ certain positive constants, $r$ depending on the boundary parameters $k_{04}$, $k_{14}$ and the viscoelastic damping coefficient $\alpha$, for all values
of the other parameters. In particular, from the results it is readily seen that the only possible accumulation points of the eigenvalues in the extended complex plane are $-1/\alpha$ and infinity.

Among the many somewhat related matters that we have not been able to address in this paper, one is the ``freedom'' in the choice of the space $\mathbb{X}$ in which to recast \eqref{301a}--\eqref{302d} as a spectral problem for a linear operator. Typically, the space $\bm{W}^2_2\left(0,1\right)\times \bm{L}_2\left(0,1\right)$ would be an alternative natural choice. Here, however, the situation is more complicated since the inner product in it induces only a seminorm and it is necessary to define $\mathbb{X}$ as a suitable subspace of $\bm{W}^2_2\left(0,1\right)\times \bm{L}_2\left(0,1\right)$.

Let us return to the initial/boundary-value problem \eqref{eq001}--\eqref{eq003d} formulated in the Introduction, which we used to motivate our investigation. With the choice of \textit{state space} as \eqref{heq01}, it is not difficult to check that the problem is equivalent to the abstract initial-value problem
\begin{equation}\label{eq700}
\dot{{x}}\left(t\right)=A{{x}}\left(t\right),\qquad A\coloneqq A_0+A_1,\qquad{{x}}\left(0\right)=x_0,
\end{equation}
for some given initial point $x_0\in\bm{D}\left(A_0\right)$, where the state ${{x}}\left(t\right)$ can be represented formally by the column vector $(\bm{w}\left(\,\cdot\,,t\right),\bm{v}\left(\,\cdot\,,t\right),\left.\left(\partial \bm{v}/\partial s\right)\left(s,t\right)\right|_{s=0},\bm{v}\left(0,t\right),\left.\left(\partial \bm{v}/\partial s\right)\left(s,t\right)\right|_{s=1},\bm{v}\left(1,t\right)\!{)}^\top$. In Section \ref{sec2}, we have disposed of the question of well-posedness of \eqref{eq700} by showing that, under our assumptions, the \textit{system operator} $A$ generates a $C_0$-semigroup of bounded linear operators on $\mathbb{X}$. An important question which remains is whether that semigroup is, in fact, holomorphic. In this connection the Riesz basis property of the root vectors (eigen- and associated vectors) of the system operator $A$ will be useful. To our knowledge this question has not been discussed for our initial/boundary-value problem.

For the problem of exponential stability of (solutions of) \eqref{eq700} it is necessary to investigate  in detail what assumptions must be placed on the quadruples $\left\{\alpha,\beta,\eta, \delta\right\}$ and $\left\{k_{02},k_{04},k_{12},k_{14}\right\}$ so that there exist constants $M,\varepsilon>0$ such that
\begin{equation*}
\left\|{{x}}\left(t\right)\right\|'_\mathbb{X}\leq M e^{-\varepsilon t}\left\| x_0\right\|'_\mathbb{X},
\end{equation*}
where we recall that $\left\|\,\cdot\,\right\|'_\mathbb{X}$ is equivalent to the energy norm for the initial/boundary-value problem \eqref{eq001}--\eqref{eq003d}. This question has been more or less settled in the papers \cite{Miloslavskii1985,Miloslavskii1983,MiloslavskiiEtAl1985,Roh1982} for the initial/boundary-value problem involving the standard boundary conditions of clamped, free, or hinged ends, and combinations thereof. For the case of generalised boundary conditions (in the sense as it is understood here) the problem is still open and requires a thorough investigation of the Riesz basis property.

There are numerous problems in mechanics and mathematical physics which give rise to initial/boundary-value problems on networks (a good overview of several examples has very recently been given in \cite{AvdoninEdward2021}). Their stability properties can be studied by using recent developments in the mathematical analysis of differential equations on geometric graphs, along the lines, for example, of \cite{MahinzaeimEtAl2021}. There we studied the exponential stability problem for a stretched Euler--Bernoulli beam (that is $\alpha=\beta=\delta=0$ and $\eta<0$) on a star graph with three identical edges. However, it is not clear whether the methods in \cite{MahinzaeimEtAl2021} can be so modified as to encompass star-shaped networks of tubes conveying fluid. We plan to examine this, and other related questions, in a subsequent paper.

\bigskip\noindent
\textbf{Acknowledgments.} This work was supported in part by the National Natural Science Foundation of China under Grant NSFC-61773277.

\bibliographystyle{plain}
\bibliography{BibLio02}

\begin{thebibliography}{10}

\bibitem{AissaEtAl2021}
A.~B. Aissa, M.~Abdelli, and A.~Duca.
\newblock Well-posedness and exponential decay for the {E}uler--{B}ernoulli
  beam conveying fluid equation with non-constant velocity and dynamical
  boundary conditions.
\newblock {\em Z.\ Angew.\ Math.\ Phys.}, 72:1--15, 2021.

\bibitem{Artamonov2000}
N.~V. Artamonov.
\newblock Stability of solutions of an equation arising in hydromechanics.
\newblock {\em Math.\ Notes}, 67:12--19, 2000.

\bibitem{AvdoninEdward2021}
S.~Avdonin and J.~Edward.
\newblock An inverse problem for quantum trees with observations at interior
  vertices.
\newblock {\em Netw.\ Heterog.\ Media}, 16:317--339, 2021.

\bibitem{BajajEtAl1980}
A.~K. Bajaj, P.~R. Sethna, and T.~S. Lundgren.
\newblock Hopf bifurcation phenomena in tubes carrying a fluid.
\newblock {\em SIAM J.\ Appl.\ Math.}, 39:213--230, 1980.

\bibitem{BanksInman1991}
H.~T. Banks and D.~J. Inman.
\newblock On damping mechanisms in beams.
\newblock {\em Trans.\ ASME J.\ Appl.\ Mech.}, 58:716--723, 1991.

\bibitem{Birkhoff1908b}
G.~D. Birkhoff.
\newblock Boundary value and expansion problems of ordinary linear differential
  equations.
\newblock {\em Trans.\ Amer.\ Math.\ Soc.}, 9:373--395, 1908.

\bibitem{Birkhoff1908a}
G.~D. Birkhoff.
\newblock On the asymptotic character of the solutions of certain linear
  differential equations containing a parameter.
\newblock {\em Trans.\ Amer.\ Math.\ Soc.}, 9:219--231, 1908.

\bibitem{Dotsenko1979}
P.~D. Dotsenko.
\newblock Intrinsic oscillations of rectilinear pipelines with liquid.
\newblock {\em Soviet Appl.\ Mech.}, 15:52--57, 1979.

\bibitem{GohbergEtAl1990}
I.~Gohberg, S.~Goldberg, and M.~A. Kaashoek.
\newblock {\em Classes of {L}inear {O}perators.\ {V}ol.\ {I}}.
\newblock Birkh\"auser, 1990.

\bibitem{Handelman1955}
G.~H. Handelman.
\newblock A note on the transverse vibration of a tube containing flowing
  fluid.
\newblock {\em Quart.\ Appl.\ Math.}, 13:326--330, 1955.

\bibitem{Holmes1977}
P.~J. Holmes.
\newblock Bifurcations to divergence and flutter in flow-induced oscillations:
  a finite dimensional analysis.
\newblock {\em J.\ Sound Vibration}, 53:471--503, 1977.

\bibitem{Kato1995}
T.~Kato.
\newblock {\em Perturbation {T}heory for {L}inear {O}perators}.
\newblock Springer, 1995.

\bibitem{Khemmoudj2021}
A.~Khemmoudj.
\newblock Stabilisation of a viscoelastic beam conveying fluid.
\newblock {\em Internat.\ J.\ Control}, 94:235--247, 2021.

\bibitem{Lyong1993}
N.~V. Lyong.
\newblock The spectral properties of a quadratic pencil of operators.
\newblock {\em Russian Math.\ Surveys}, 48:180--182, 1993.

\bibitem{MahinzaeimEtAl2021}
M.~Mahinzaeim, G.~Q. Xu, and H.~E. Zhang.
\newblock On the stability of the stretched {E}uler--{B}ernoulli beam on a
  star-shaped graph.
\newblock arXiv:2106.00129v3 [math.AP], June 2021.

\bibitem{Miloslavskii1981}
A.~I. Miloslavskii.
\newblock Spectral properties of a class of quadratic operator pencils.
\newblock {\em Funct.\ Anal.\ Appl.}, 15:142--144, 1981.

\bibitem{Miloslavskii1983}
A.~I. Miloslavskii.
\newblock Foundation of the spectral approach in nonconservative problems of
  the theory of elastic stability.
\newblock {\em Funct.\ Anal.\ Appl.}, 17:233--235, 1983.

\bibitem{Miloslavskii1985}
A.~I. Miloslavskii.
\newblock Stability of certain classes of evolution equations.
\newblock {\em Sib.\ Math.\ J.}, 26:723--735, 1985.

\bibitem{Miloslavskii1991}
A.~I. Miloslavskii.
\newblock Instability spectrum of an operator bundle.
\newblock {\em Math.\ Notes}, 49:391--395, 1991.

\bibitem{Movchan1965}
A.~A. Movchan.
\newblock On one problem of stability of a pipe with a fluid flowing through
  it.
\newblock {\em J.\ Appl.\ Math.\ Mech.}, 29:902--904, 1965.

\bibitem{Naimark1967}
M.~A. Naimark.
\newblock {\em Linear {D}ifferential {O}perators.\ {P}art {I}: {E}lementary
  {T}heory of {L}inear {D}ifferential {O}perators}.
\newblock Frederick Ungar, 1967.

\bibitem{NesterovAkulenko2008}
S.~V. Nesterov and L.~D. Akulenko.
\newblock Spectrum of transverse vibrations of a moving rod.
\newblock {\em Dokl.\ Phys.}, 53:265--269, 2008.

\bibitem{Paidoussis2014}
M.~P. Pa\"idoussis.
\newblock {\em Fluid-{S}tructure {I}nteractions: {S}lender {S}tructures and
  {A}xial {F}low. {V}ol.\ {1}}.
\newblock Elsevier, 2014.

\bibitem{PaidoussisIssid1974}
M.~P. Pa\"idoussis and N.~T. Issid.
\newblock Dynamic stability of pipes conveying fluid.
\newblock {\em J.\ Sound Vibration}, 33:267--294, 1974.

\bibitem{Pazy1983}
A.~Pazy.
\newblock {\em Semigroups of {L}inear {O}perators and {A}pplications to
  {P}artial {D}ifferential {E}quations}.
\newblock Springer, 1983.

\bibitem{Pivovarchik2005}
V.~Pivovarchik.
\newblock Necessary conditions for stability of elastic pipe conveying liquid.
\newblock {\em Methods Funct.\ Anal.\ Topology}, 11:270--274, 2005.

\bibitem{Pivovarchik1992}
V.~N. Pivovarchik.
\newblock On the total algebraic multiplicity of the spectrum in the right
  half-plane for a class of quadratic operator pencils.
\newblock {\em St.\ Petersburg Math.\ J.}, 3:447--454, 1992.

\bibitem{Pivovarchik1993}
V.~N. Pivovarchik.
\newblock Necessary conditions for gyroscopic stabilization in a problem of
  mechanics.
\newblock {\em Math.\ Notes}, 53:622--627, 1993.

\bibitem{Pivovarchik1994}
V.~N. Pivovarchik.
\newblock A lower bound of the instability index in the vibration problem for
  an elastic fluid-conveying pipe.
\newblock {\em Russian J.\ Math.\ Phys.}, 2:267--272, 1994.

\bibitem{Roh1982}
H.~R{\"o}h.
\newblock Dissipative operators with finite dimensional damping.
\newblock {\em Proc.\ Roy.\ Soc.\ Edinburgh Sect.\ A}, 91:243--263, 1982.

\bibitem{Russell1986}
D.~L. Russell.
\newblock Mathematical models for the elastic beam and their control-theoretic
  implications.
\newblock In {\em Semigroups, {T}heory and {A}pplications.\ {V}ol.\ {II}},
  volume 152 of {\em Pitman Res.\ Notes Math.\ Ser.}, pages 177--216. Longman
  Scientific \& Technical, 1986.

\bibitem{MiloslavskiiEtAl1985}
V.~N. Zefirov, V.~V. Kolesov, and A.~I. Miloslavskii.
\newblock A study of the natural frequencies of a rectilinear pipe.
\newblock {\em Izv.\ Akad.\ Nauk SSSR Mekh.\ Tverd.\ Tela}, 1:179--188, 1985.
\newblock (In Russian).

\bibitem{ZhuEtAl2000}
W.~D. Zhu, B.~Z. Guo, and Jr. Mote, C.~D.
\newblock Stabilization of a translating tensioned beam through a pointwise
  control force.
\newblock {\em Trans.\ ASME J.\ Dyn.\ Sys., Meas., Control.}, 122:322--331,
  2000.

\end{thebibliography}

\end{document}